\documentclass{iopart}
\usepackage{diagbox}%
\usepackage{iopams}
\eqnobysec
\usepackage{graphicx}
\usepackage{epstopdf}
\usepackage{color}
\usepackage{amsfonts}
\usepackage{amssymb}
 \expandafter\let\csname equation*\endcsname\relax
 \expandafter\let\csname endequation*\endcsname\relax
\usepackage{amsmath}
\usepackage{amsxtra}
\usepackage{amsthm}
\usepackage{latexsym}

\newtheorem{lemma}{Lemma}[section]
\newtheorem{theorem}[lemma]{Theorem}

\newtheorem{proposition}[lemma]{Proposition}

\newtheorem{Assumption}[lemma]{Assumption}

\def\p{\partial}
\def\d{\delta}

\def\p{\partial}
\def\d{\delta}

\def\l{\langle}
\def\r{\rangle}
\def\N{\mathcal N}
\def\R{\mathcal R}
\def\D{\mathcal D}
\def\X{\mathcal X}
\def\Y{\mathcal Y}

\def\d{\delta}

\def\X{\mathcal X}
\def\Y{\mathcal Y}

\begin{document}
\jl{5}
\title[]
%{A novel class of reconstruction schemes with non-smooth convex penalty for nonlinear inverse problems in Banach spaces}
{An iteration regularizaion method with general convex penalty for nonlinear inverse problems in Banach spaces}

\author{Jing Wang$^1$, Wei Wang$^2$\begin{footnote} {Corresponding author}\end{footnote}, Bo Han$^3$}

\medskip

\address{$^1$School of Mathematical Sciences, Heilongjiang University, Harbin, Heilongjiang 150080, China (jingwangmath@hlju.edu.cn)}
\address{$^2$College of Mathematics, Physics and Information Engineering, Jiaxing University, Zhejiang 314001, China (weiwang\_math@126.com) }
\address{$^3$Department of Mathematics, Harbin Institute of Technology, Harbin, Heilongjiang 150001, China (bohan@hit.edu.cn)}
%\medskip

%\eads{}

%%%%%%%%%%%%%%%%%%%%%%%%%%%%%%%%%%%%%%%%%%%%%%%%%%%%%%%%%%%%%%%%%%%%%%%%%%%%%%%%

%\keywords{nonlinear inverse problems, the iteratively regularized Gauss-Newton method, heuristic selection rule, a posteriori error estimates, convergence}

\begin{abstract}
In this paper, we discuss the construction, analysis and implementation of a novel iterative regularization scheme with general convex penalty term for nonlinear inverse problems in Banach spaces based on the homotopy perturbation technique, in an attempt to detect the special features of the sought solutions such as sparsity or piecewise constant.
By using tools from convex analysis in Banach spaces, we provide a detailed convergence and stability results for the presented algorithm.
Numerical simulations for one-dimensional and two-dimensional parameter identification problems are performed to validate that our approach is competitive in terms of reducing the overall computational time in comparison with the existing Landweber iteration with general convex penalty.
\end{abstract}

%\ams{65J15, 65J20, 47H17}

%\submitto{\IP}

%\maketitle

%
\oddsidemargin 15mm
\evensidemargin 15mm
\topmargin  0 pt
%\textheight=8.6 true in
%\textwidth=5.2 true in

%\vskip 0.5 cm

%%%%%%%%%%%%%%%%%%%%%% -- SECTION 1 -- %%%%%%%%%%%%%%%%%%%%%%%%%%%%%%%%%%%%%%%%
%%%%%%%%%%%%%%%%%%%%%% -- SECTION 1 -- %%%%%%%%%%%%%%%%%%%%%%%%%%%%%%%%%%%%%%%%
%\section{Introduction}\label{s1}
%
\section{Introduction}
In this paper, we will consider the nonlinear ill-posed operator equation
\begin{equation}\label{Nonlinear}
F(x) = y,
\end{equation}
where $F:D(F)\subset\mathcal{X}\rightarrow\mathcal{Y}$ is a nonlinear operator between the Banach spaces $\mathcal{X}$ and $\mathcal{Y}$
with norms $\|\cdot\|$, whose topological dual spaces are denoted by $\X^*$ and $\Y^*$, respectively.
Instead of the right hand side $y$, only a noisy data $y^\delta$ is available such that
\begin{equation}\label{noise}
\|y-y^\delta\|\leq\delta,
\end{equation}
with a small known noise level $\delta>0$.
Due to the ill-posedness of equation \eqref{Nonlinear}, a direct inversion of noise-contaminated data $y^\delta$ would not lead to a
meaningful solution.
Consequently, to find the stable and desired approximations of solutions of \eqref{Nonlinear}, we have to apply some regularization strategy.
Tikhonov regularization is certainly the most popular stabilization approach for solving nonlinear ill-posed problems \cite{Engl,Tikhonov1}
and  generalized in Banach spaces \cite{Banach}.
The minimization of Tikhonov-type functional is usually realized via optimization schemes, in which the good choice of the regularization parameter is crucial for the quality of the reconstructed solution, which often leads to increasing numerical stabilities and costs.
On the contrary, due to the straightforward implementation, iterative regularization methods \cite{Iterative} seem to be a promising and attractive alternative, in which the iterative steps plays the role of the regularization parameters.

Here the focus is on the generalization of the Landweber iteration method regarding Banach spaces.
For given parameter $r>1$, by making use of a gradient method for solving the minimization problem
\[
\min \frac{1}{r}\|F(x)-y^\delta\|^r,
\]
we therefore consider the following iteration
\begin{equation}\label{landr1}
\xi_{n+1} = \xi_{n} -\mu_{n}F'(x_{n})^* J_r^{\Y}(F(x_{n})-y^\d),
~~~~~~~~x_{n+1} = J_{s^*}^*(\xi_{n+1}),
\end{equation}
together with suitably chosen step length $\mu_n$, where $F'(x_{n})^*:\Y^*\to\X^*$ denotes the adjoint of $F'(x_{n})$, $J_r^{\Y}:\Y\to\Y^*$ and $J_{s^*}^{*}:\X^*\to\X$ denote the corresponding duality mapping with gauge function $t \mapsto t^{r-1}$ and $t \mapsto t^{s^*-1}$ respectively.
The choice of the parameter $s^*\in(1,2]$ is determined by the supposed smoothness of the dual space $\X^*$.
Starting with \cite{sls06} for linear problems, many publications have been concerned with an iteration of this type for nonlinear problems, see \cite{hk10,JinWang,kss09}.
In \cite{JinWang}, introducing a uniformly convex penalty $\Theta: \X\rightarrow(-\infty, \infty]$ chosen with desired features, the Landweber-type iteration with general convex penalty (named by LICP later) was proposed:
\begin{equation}\label{LICP}
\xi_{n+1}= \xi_{n} -\mu_{n}F'(x_{n})^* J_r^{\Y}(F(x_{n})-y^\d),~~~~~~
x_{n+1}= \nabla \Theta^*(\xi_{n+1}),
\end{equation}
where $\Theta^*:\X^*\to (-\infty, \infty]$  is the convex conjugate of $\Theta$ and $\nabla\Theta^*:\X^*\to \X$ denotes its gradient.
The function $\Theta$ can be chosen as the hybrid terms combining two very powerful features: $L^2+L^1$ function known to promote sparsity and $L^2+TV$ function  allowing for detecting the sharp edges.
Iteration \eqref{LICP} can be interpreted as linearized Bregman iteration see \cite{lsw14,yog08} in the linear case.
A general convergence analysis and regularization results on \eqref{LICP} is given in \cite{JinWang,ms16} under the termination with the discrepancy principle.

Homotopy perturbation iteration for nonlinear ill-posed problems in Hilbert spaces was first constructed by Li Cao, Bo Han and Wei Wang in \cite{Cao1,Cao2}.
Its essential idea is to introduce an embedding homotopy parameter and combine the traditional perturbation method with the homotopy technique.
Using the notation $T_n:=F'(x_{n})$, N-order homotopy perturbation iteration method can be formulated by
\begin{equation}\label{hpn}
 x_{n+1}=x_n-\sum_{j=1}^{N}(I- T_n^*T_n)^{j-1}T_n^*(F(x_n)-y^\delta).
\end{equation}
It is noteworthy that (\ref{hpn}) also can be explained as the $N$-steps classical Landweber iteration for solving the linearized problem \cite{jin11}: $F(x_n)+T_n(x-x_n)=y^\d.$
With the one-order approximation truncation ($N=1$), it can yield the classical Landweber iteration \cite{Hanke}:
\begin{equation}\label{Landweber}
x_{n+1}=x_{n}-T_n^{*}(F(x_{n})-y^\d).
\end{equation}
With the two-order approximation truncation ($N=2$), the homotopy perturbation iteration \cite{Cao1} can be obtained:
\begin{equation}\label{NewLandweber}
x_{n+1}=x_{n}-T_n^{*}(2I-T_nT_n^{*})(F(x_{n})-y^\d).
\end{equation}
It is shown that only half-time for \eqref{NewLandweber} is needed with the same accuracy compared with \eqref{Landweber}.
Subsequently, it was successfully applied to the well log constrained seismic waveform inversion \cite{Fu}.
%and the electrical impedance tomography problem \cite{Jing}.
Nevertheless, both \eqref{Landweber} and \eqref{NewLandweber} mainly restricted to the case of quadratic penalty terms may be no longer available for detecting the specific solutions with the discontinuity of sharp points or edges.

Inspired by the homotopy perturbation iteration in Hilbert space, in this paper we propose a novel iteration regularization method with general uniformly convex penalty for nonlinear inverse problems in Banach spaces.
The approach (homotopy perturbation iteration with general convex penalty, named by HPICP later) generalizes the two-order homotopy perturbation iteration \eqref{NewLandweber} in Banach spaces, in which the duality mappings
$J_r^{\Y}:\Y\to\Y^*, J_{s^*}^{*}:\X^*\to\X$ with $1<r,s^*<\infty$  are used and  general uniformly convex penalty $\Theta: \X\rightarrow(-\infty, \infty]$ is introduced.
For simplicity of exposition, we here write the notation $r_n := J_r^{\Y}(F(x_{n})-y^\d)$ and the proposed HPICP method can be formulated as follows :
\begin{equation}\label{HPICP}
\xi_{n+1}= \xi_{n} -\mu_n T_n^*\left(2r_n-\nu_nJ_{r}^{\Y}(T_nJ_{s^*}^{*}T_n^* r_n)\right),~~~~~~
x_{n+1}=\nabla \Theta^*(\xi_{n+1}).
\end{equation}
with  suitably chosen step length $\mu_n,\nu_n$.
In contrast to LICP method \cite{JinWang}, our proposed approach has the advantage of improving the calculation speed by strongly reducing the iteration numbers.
As an iterative regularization method, the discrepancy principle is used to terminate the iteration.
Due to the non-smooth convex penalty, which may include $L^1$ penalty or $TV$  penalty, the iteration can produce good results in applications,
where the sought solution is sparse or discontinuous. Moreover, iterative regularization in Banach spaces can be used to the non-Gaussian noisy data.
We expect that our method can become favorable by using other accelerated versions.

The outline of this paper is as follows.
In Section 2, we give some preliminary results from convex analysis in Banach spaces.
In Section 3,  we present the detailed convergence analysis and regularization results of our method combined with the discrepancy principle as stooping rule.
In section 4 we report some numerical simulations to test the performance of the method.
Finally, a short conclusion is drawn in Section 5.

\section{Preliminaries}
Throughout this paper $\X$ is supposed to be uniformly smooth and uniformly convex, hence it is reflexive and the dual $\X^{*}$ has the same properties \cite{Duality}.
Let $1<s,s^*<\infty$ denote conjugate exponents, i.e., $1/s+1/s^*=1$.
For any $x\in \X$ and $\xi\in \X^*$, we write $\langle x,\xi\rangle=\langle \xi,x\rangle$ for the duality pairing.
We use $A: \X \to \Y$ to denote a bounded linear operator and $A^*:\Y^*\to \X^*$ to denote its adjoint, i.e. $\langle A^*\varsigma,x\rangle=\langle \varsigma,Ax\rangle$ for any $x\in \X, \varsigma\in \Y^*$, and $\|A\|=\|A^*\|$ for the operator norm of $A$.
Let $\N(A)=\{x\in \X: A x=0\}$ be the null space of $A$, and let
$$\N(A)^\perp :=\{\xi\in \X: \l \xi, x\r =0, {\rm~for~all~} x\in \N(A)\}.$$
be the annihilator of $\N(A)$.

On a Banach space $\Y$, we consider the convex function $x\to \|x\|^r/r~(1<r<\infty)$.
Its subdifferential at $x$ is given by
$$J_r^\Y(x):=\left\{\xi\in \Y: \|\xi\|=\|x\|^{r-1}~{\rm and}~\l \xi, x\r=\|x\|^r\right\},$$
which gives the duality mapping $J_r^\Y: \Y \to 2^{\Y^*}$ of $\Y$ with gauge function $t\mapsto t^{r-1}$.
It is well known that the duality mapping $J_r^\Y$ ($1<r<\infty$) is single valued and uniformly continuous on bounded sets if $\Y$ is uniformly smooth.
It is an in general nonlinear set-value mapping.
For $s^*>1$ with $1/s+1/s^*=1$ we denote by $J_{s^*}^*:\X^*\to \X$ the duality mapping of the $\X^*$ with gauge function $t\mapsto t^{s^*-1}$.

Given a convex function $\Theta: \X \to (-\infty, \infty]$, we use $D(\Theta):=\{x\in \X: \Theta(x)<+\infty\}$ to denote its effective domain. It is called proper if $D(\Theta)\ne \emptyset$. The subdifferential of $\Theta$ at $x\in \X$ is defined as
\begin{equation}\label{SubD}
\p \Theta(x):=\{\xi\in\X^*: \Theta(z)-\Theta(x)-\l \xi, z-x\r \ge 0, \rm{~for~all~} z\in \X\}.
\end{equation}
The subdifferential mapping $\p\Theta:\X\to 2^{\X^*}$ is multi-valued and we set
$D(\p\Theta):=\{x\in D(\Theta): \p \Theta(x)\ne \emptyset\}.$

A proper convex function $\Theta: \X\to (-\infty, \infty]$ is called uniformly convex if there is a continuous increasing function $h: [0, \infty]\to [0, \infty]$, with the property that $h(t)=0$ implies $t=0$, such that
$$\Theta(\gamma \bar x+(1-\gamma)x) +\gamma (1-\gamma) h(\|\bar x-x\|) \le \gamma \Theta(\bar x) +(1-\gamma) \Theta(x).$$
for all $x, \bar x\in \X$ and all $\gamma\in [0,1]$.
If $h$ can be taken as $h(t)=c_0t^p$ for some $c_0>0$ and $p\ge 2$, then $\Theta$ is called $p$-convex.
Any uniformly convex mapping is strictly convex.

In the convex analysis, the Legendre-Fenchel conjugate is an important notation.
Given a proper, lower semi-continuous, convex function $\Theta: \X \to (-\infty, \infty]$, its Legendre-Fenchel conjugate $\Theta^*:\X^*\to[-\infty, \infty]$ is defined by $$\Theta^*(\xi):=\sup_{x\in \X} \left\{\l\xi, x\r -\Theta(x)\right\}, \quad \forall\xi\in \X^*.$$
It is well known that $\Theta^*$ is also proper, lower semi-continuous and convex.
And as an immediate consequence of the definition, we will have the following lemma.
\begin{lemma}
For arbitrary $x\in\X,\xi\in\X^*$, Young-Fenchel inequality holds as follows:
$$\l \xi, x\r \leq \Theta(x) +\Theta^*(\xi)$$
and
\begin{equation}\label{L-Fequal}
 \xi\in \p \Theta(x)  \Leftrightarrow x\in\p \Theta^*(\xi) \Leftrightarrow \l \xi, x\r=\Theta(x) +\Theta^*(\xi).
\end{equation}
\end{lemma}

We in Banach spaces introduce the Bregman distance with respect to convex function $\Theta$, which for any $x\in D(\p\Theta(x))$ and $\xi\in \p \Theta(x)$ is given by
$$D_\xi \Theta(z,x):=\Theta(z)-\Theta(x)-\l \xi, z-x\r, \qquad \forall z\in \X.$$
It is clear that the Bregman distance is non-negative and it holds $D_\xi \Theta(x,x)= 0$.
Bregman distance can be used to obtain important information under the Banach space norm when $\Theta$ has stronger convexity.

\begin{lemma}[\cite{Zalinescu},Corollary 3.5.11]
Let $\Theta: \X\to  (-\infty, \infty]$ is proper, lower semi-continuous and p-convex for some $p\geq2$. Then
 \begin{itemize}
   \item[1)]  there exists a constant $c_0>0$ such that
   \begin{equation}\label{pconv}
D_\xi \Theta(\bar x,x) \ge c_0 \|\bar x-x\|^p,~~~~\forall\bar x\in \X~{\rm and}~ \xi\in \p \Theta(x).
\end{equation}
   \item[2)]$D(\Theta^*)=\X^*$, $\Theta^*$ is Fr\'{e}chet differentiable and its gradient $\nabla \Theta^*: \X^*\to \X$ satisfies
\begin{equation}\label{thgradient}
\|\nabla \Theta^*(\xi_1)-\nabla \Theta^*(\xi_2) \|\le \left(\frac{\|\xi_1-\xi_2\|}{2c_0}\right)^{\frac{1}{p-1}},~~\forall\xi_1, \xi_2\in \X^*.
\end{equation}
 \end{itemize}
 \end{lemma}
In addition, by the subdiffierential calculus there also holds
\begin{equation}\label{Eual}
x=\nabla \Theta^*(\xi) \Longleftrightarrow x =\arg \min_{x\in \X} \left\{ \Theta(x) -\l \xi, x\r\right\}.
\end{equation}

\section{The method and its convergence}
In this section we first formulate the novel iteration regularization method with the general uniformly convex penalty terms.
And then we present the detailed convergence analysis.
Throughout this section we will assume that $\X$ is $s$-convex for some $s\in[2,\infty)$, $(1/s+1/s^*=1)$ and $\Theta: \X\to (-\infty, \infty]$ is a proper, lower semi-continuous, p-convex function with $p\geq2$.
Assuming that $\Y$ is uniformly smooth Banach space so that the duality mapping $J_r^\Y:\Y \to \Y^*$ is single-valued and continuous for each $1<r<\infty$.
By picking $x_0\in \D(\partial\Theta)$ and $\xi_0\in\partial\Theta(x_0)$ as the initial guess, we define $x^\dag$ to be the solution of (\ref{Nonlinear}) with the property
\begin{equation}\label{xdag}
 D_{\xi_0} \Theta(x^{\dag},x_0) := \min_{x\in \D(\Theta)\cap \D }\left\{D_{\xi_0} \Theta(x,x_0) : F(x) = y\right\}.
\end{equation}
We are interested in developing algorithms to find the solution $x^\dag$ of (\ref{Nonlinear}).
We will need to impose the following conditions on the nonlinear operator $F$ where $B_\rho(x_0):=\{x\in \X: \|x-x_0\|\le \rho\}$.
\begin{Assumption}\label{A1}
\begin{itemize}
\item[(a)] There is $\rho > 0$ such that $B_{2\rho}(x_0)\subset \D(F)$ and (\ref{Nonlinear}) has a solution in $B_\rho(x_0)\cap \D(\Theta)$.

\item[(b)] Operator $F$ is weakly closed on $\D(F)$ and is Fr\'{e}chet differentiable on $B_{2\rho}(x_0)$, and  $F':x\to F'(x)$ is continuous on $B_{2\rho}(x_0)$.

\item[(c)] Fr\'{e}chet operator $F'$ is locally uniformly bounded so that  $$\|F'(x)\|\le B_0, \quad \forall x\in B_{2\rho}(x_0).$$

\item[(d)] There exists $0\le \eta<1$ such that the tangential cone condition holds
$$\|F(x)-F(\bar x) -F'(\bar x) (x-\bar x)\|\le \eta \|F(x) -F(\bar x)\|,\quad \forall x, \bar x \in B_{2\rho}(x_0).$$
\end{itemize}
\end{Assumption}
When $\X$ is a reflexive Banach space, by using the p-convexity and the weakly lower semi-continuity of $\Theta$ together with the weakly closedness of $F$, it is standard to show that $x^\dag$ exists.
The following result shows that $x^\dag$ is in fact uniquely defined, and more detailed proof can be seen in \cite{JinWang}.

\begin{lemma}\label{lemxdag}
Let $\X$ be reflexive and $F$ satisfy Assumption \ref{A1}.
If $x^\dag\in B_\rho(x_0)\cap \D(\Theta)$, then $x^\dag$ is the unique solution of (\ref{Nonlinear}) in $B_{2\rho}(x_0)\cap \D(\Theta)$ satisfying (\ref{xdag}).
\end{lemma}

For the situation that the data contains noise, we may imitate \eqref{HPICP} to define an iterative sequence $\{(x_{n}^\d,\xi_{n}^\d)\}$ in $\X\times\X^*$.
For simplicity of the presentation we set $T_n^\d:=F'(x_{n}^\d)$ and $r_n^\d:= J_r^{\Y}(F(x_{n}^\d)-y^\d)$.
We denote the initial guess by $x_0^\d:=x_0\in \D(\partial\Theta)$ and $\xi_0^\d:=\xi_0\in\partial\Theta(x_0)$.
Once we have $\{(x_{n}^\d,\xi_{n}^\d)\}$, we may define $\{(x_{n+1}^\d,\xi_{n+1}^\d)\}$ by
\begin{equation}\label{HPICPdelta}
\left\{
\begin{array}{rl}
\xi_{n+1}^\d &= \xi_{n}^\d -\mu_{n}^\d {T_n^\d}^*\left(2r_n^\d-\nu_n^\d J_{r}^{\Y}(T_n^\d J_{s^*}^{*}{T_n^\d}^* r_n^\d)\right),\cr
x_{n+1}^\delta &= \arg \min\limits_{x\in \X} \left\{ \Theta(x) -\l \xi_{n+1}^\d, x\r\right\},
\end{array}
\right.
\end{equation}
with a proper choice of the step size $\mu_n^\d$.
Note that by using (\ref{Eual}) one can see that
$$x_{n+1}^\d =\arg \min\limits_{x\in \X} \left\{ \Theta(x) -\l \xi_{n+1}^\d, x\r\right\} \Longleftrightarrow x_{n+1}^\d=\nabla \Theta^*(\xi_{n+1}^\d),$$
which will be used in the forthcoming theoretical analysis.

In case of noisy data, the iteration procedure \eqref{HPICPdelta} has to be coupled with a stopping rule in order to act as a regularization method.
We will employ the discrepancy principle as a stopping rule, which determines the stopping index $n_\d=n_\d(\d,y^\d)$ by
$$\|F(x_{n_\d}^\d)-y^\d\| \leq \tau \d \leq \|F(x_{n}^\d)-y^\d\|,~~~~0\leq n<n_\d,$$
for some sufficiently large $\tau>0$, i.e., $x_{n_\d}^\d$ is the calculated approximate solution.

Next we will show that \eqref{HPICPdelta} has well convergence under the discrepancy principle.
In the following proposition we will first prove monotonicity of the errors.
\begin{proposition}[Error analysis]\label{L5.1}
Let Assumption \ref{A1} hold with $0\leq\eta<1$ and let $\Theta: \X\to  (-\infty, \infty]$ be a  proper, lower semi-continuous, p-convex function with $p\geq2$ satisfying (\ref{pconv}) for some $c_0>0$.
Assume that
\begin{equation}\label{5.11.1}
D_{\xi_0} \Theta (x^\dag, x_0) \le c_0 \rho^p.
\end{equation}
Let $\{(x_{n}^\d,\xi_{n}^\d)\}$ be defined by iteration (\ref{HPICPdelta}) with
\begin{eqnarray*}
\nu_{n}^\d &= \frac{\|r_n^\d\|}{\|J_{r}^{\Y}(T_n^\d J_{s^*}^{*}{T_n^\d}^* r_n^\d)\|}, \quad  \mu_n^\d= \mu_0 B_0^{-p}\|F(x_{n}^\d)-y^\d\|^{p-r},
\end{eqnarray*}
and the iteration is terminated by the discrepancy principle
\[
\|F(x_{n_\d}^\d)-y^\d\| \leq \tau \d \leq \|F(x_{n}^\d)-y^\d\|, \quad 0\leq n<n_\d
\]
with $\tau>1$ and $\mu_0>0$ satisfies
$$c_1:=1-3\eta-\frac{3(\eta+1)}{\tau}-3\left(\frac{3\mu_0}{2c_0}\right)^{\frac{1}{p-1}}>0.$$
Then $n_\d<\infty$ and $x_{n}^\d\in B_{2\rho}(x_0)$ for all $n\ge 0$.
Moreover, for any solution $\hat x$ of $F(x)=y$ in $B_{2\rho}(x_0)\cap D(\Theta)$ and all $n$ there hold
\begin{eqnarray}
&&D_{\xi_{n+1}^\d}\Theta(\hat x, x_{n+1}^\d)  \le D_{\xi_n^\d}\Theta(\hat x, x_n^\d),\label{eq5.10}\\
&& c_1 \mu_{n}^\d \|F(x_{n}^\d)-y^\d\|^r \le D_{\xi_n^\d}\Theta(\hat x, x_n^\d)
-D_{\xi_{n+1}^\d}\Theta(\hat x, x_{n+1}^\d). \label{eq5.11}
\end{eqnarray}
\end{proposition}
\begin{proof}
From the definition of Bregman distance and together with using \eqref{L-Fequal}, it follows that
\begin{equation*}
\begin{aligned}
D_{\xi_{n+1}^\d}\Theta(\hat x,& x_{n+1}^\d) -D_{\xi_{n}^\d}\Theta(\hat x, x_{n}^\d)\\
&= \Theta(x_{n}^\d) -\Theta(x_{n+1}^\d) -\l \xi_{n+1}^\d, \hat x -x_{n+1}^\d\r + \l \xi_{n}^\d, \hat x -x_{n}^\d\r\\
&=\Theta^*(\xi_{n+1}^\d) -\Theta^*(\xi_{n}^\d) -\l \xi_{n+1}^\d- \xi_{n}^\d, \hat{x}\r.
\end{aligned}
\end{equation*}
By introducing $x_{n}^\d=\nabla \Theta^*(\xi_{n}^\d)$, we can write
\begin{equation*}
\begin{aligned}
D_{\xi_{n+1}^\d}\Theta(\hat x,& x_{n+1}^\d) -D_{\xi_{n}^\d}\Theta(\hat x, x_{n}^\d)\\
&= \Theta^*(\xi_{n+1}^\d) -\Theta^*(\xi_{n}^\d)  -\l \xi_{n+1}^\d-\xi_{n}^\d, \nabla \Theta^*(\xi_{n}^\d)\r
+\l \xi_{n+1}^\d-\xi_{n}^\d, x_{n}^\d-\hat{x}\r.
\end{aligned}
\end{equation*}
Since $\Theta$ is p-convex, we may use (\ref{thgradient}) to obtain
\begin{equation*}
\begin{aligned}
& \Theta^*(\xi_{n+1}^\d) -\Theta^*(\xi_{n}^\d) -\l \xi_{n+1}^\d-\xi_{n}^\d, \nabla \Theta^*(\xi_{n}^\d)\r\\
&= \int_0^1 \l \xi_{n+1}^\d -\xi_{n}^\d, \nabla \Theta^* (\xi_{n}^\d +t (\xi_{n+1}^\d-\xi_{n}^\d)) -\nabla \Theta^*(\xi_{n}^\d)\r d t\\
&\le \|\xi_{n+1}^\d-\xi_{n}^\d\| \int_0^1  \| \nabla \Theta^* (\xi_{n}^\d +t (\xi_{n+1}^\d-\xi_{n}^\d)) -\nabla \Theta^*(\xi_{n}^\d)\| dt\\
&\le (2c_0)^{1-p^*} \|\xi_{n+1}^\d-\xi_{n}^\d\|^{p^*},
\end{aligned}
\end{equation*}
where $p^*$ is the number conjugate to $p$, i.e., $1/p+1/p^*=1$.

By the definition of $\xi_{n+1}^\d$ in iteration (\ref{HPICPdelta}) we then have
\begin{equation}\label{1111}
\begin{aligned}
D_{\xi_{n+1}^\d}\Theta(\hat x, x_{n+1}^\d)-D_{\xi_{n}^\d}\Theta(\hat x, x_{n}^\d)\le (2 c_0)^{1-p^*}\|\xi_{n+1}^\d-\xi_{n}^\d\|^{p^*}
+\l \xi_{n+1}^\d-\xi_{n}^\d, x_{n}^\d-\hat{x}\r,
\end{aligned}
\end{equation}
where
\[
\xi_{n+1}^\d-\xi_{n}^\d = -\mu_{n}^\d {T_n^\d}^*\left(2 r_n^\d-\nu_n^\delta J_{r}^{\Y}({T_n^\d}J_{s^*}^{*}{T_n^\d}^* r_n^\d)\right).
\]
By virtue of the property of the duality mapping $J_r^{\Y}$, we have
\begin{equation}
\|J_{r}^{\Y}({T_n^\d}J_{s^*}^{*}{T_n^\d}^* r_n^\d)\|\leq \|{T_n^\d}J_{s^*}^{*}{T_n^\d}^*r_n^\d\|^{r-1} \quad {\rm and}\quad  \|r_n^\d\|\leq\|F(x_n^\d)-y^\d\|^{r-1} .  %\leq \|F(x_n^\d)-y^\d\|^{(r-1)^2}
\end{equation}
Moreover, according to the scaling condition in Assumption \ref{A1}(c) (i.e.,  $\|{T_n^\d}\|\leq B_0$), we have by taking $\nu_n^\d \leq  \frac{\|r_n^\d\|}{\|J_{r}^{\Y}({T_n^\d}J_{s^*}^{*}{T_n^\d}^* r_n^\d)\|}$ that
\begin{equation}
\nu_n^\delta \|{T_n^\d}^*J_{r}^{\Y}({T_n^\d}J_{s^*}^{*}{T_n^\d}^* r_n^\d)\|\leq B_0\|r_n^\d\|.
\end{equation}
Therefore,
\begin{equation*}
\begin{aligned}
\|\xi_{n+1}^\d-\xi_{n}^\d\|^{p^*}
= (\mu_{n}^\d )^{p^*}\|2{T_n^\d}^* r_n^\d-\nu_n^\delta {T_n^\d}^*J_{r}^{\Y}({T_n^\d}J_{s^*}^{*}{T_n^\d}^* r_n^\d)\|^{p^*}
\leq (3B_0)^{p^*}(\mu_{n}^\d )^{p^*}\|r_n^\d\|^{p^*}.
\end{aligned}
\end{equation*}
Furthermore, we estimate
\begin{equation*}
\begin{aligned}
\l \xi_{n+1}^\d-\xi_{n}^\d, x_{n}^\d-\hat{x}\r
 &= - \mu_n^\d \l 2 r_n^\d-\nu_n^\delta J_{r}^{\Y}({T_n^\d}J_{s^*}^{*}{T_n^\d}^* r_n^\d), {T_n^\d}(x_{n}^\d-\hat{x})\r   \\
 &= - 2\mu_n^\d \l r_n^\d, {T_n^\d}(x_{n}^\d-\hat{x})\r + \mu_n^\d \nu_n^\d\l J_r^{\Y}({T_n^\d}J_{s^*}^{*}{T_n^\d}^*r_n^\d),{T_n^\d}( x_{n}^\d-\hat{x})\r,
\end{aligned}
\end{equation*}
where
\begin{equation*}
\begin{aligned}
 -\l r_n^\d, -{T_n^\d}(\hat{x}-x_{n}^\d)\r
& =  -\l r_n^\d, F(\hat{x})-F(x_n^\d)-{T_n^\d}(\hat{x}-x_{n}^\d)\r - \l r_n^\d, F(x_n^\d)-y^\d+y^\d-y\r \\
& \leq \eta\|r_n^\d\|\|F(\hat{x})-F(x_n^\d)\| - \l r_n^\d, F(x_n^\d)-y^\d\r + \|r_n^\d\|\delta\\
 & \leq (1+\eta)\|F(x_n^\d)-y^\d\|^{r-1}\delta - (1-\eta) \|F(x_n^\d)-y^\d\|^r
\end{aligned}
\end{equation*}
and
\begin{equation*}
\begin{aligned}
\nu_n^\d\l J_r^{\Y}({T_n^\d}J_{s^*}^{*}{T_n^\d}^*r_n^\d),{T_n^\d}( x_{n}^\d-\hat{x})\r
&\leq \nu_n^\d\|J_r^{\Y}({T_n^\d}J_{s^*}^{*}{T_n^\d}^*r_n^\d)\|\|{T_n^\d}( x_{n}^\d-\hat{x})\|\\
&\leq(1+\eta) \|F(x_n^\d)-y^\d\|^{r-1} \|F(\hat{x})-F(x_n^\d)\|\\
&\leq(1+\eta) \|F(x_n^\d)-y^\d\|^{r-1} (\d+\|F(x_n^\d)-y^\d\|).
\end{aligned}
\end{equation*}
Hence, by the stopping rule we have
\begin{equation*}
\begin{aligned}
\l \xi_{n+1}^\d-\xi_{n}^\d, x_{n}^\d-\hat{x}\r
& \leq -\mu_n^\d (1-3\eta) \|F(x_n^\d)-y^\d\|^{r} +3(1+\eta)\mu_n^\d \|F(x_n^\d)-y^\d\|^{r-1}\delta\\
& \leq  -\mu_n^\d (1-3\eta-3(1+\eta)/\tau) \|F(x_n^\d)-y^\d\|^{r}.
\end{aligned}
\end{equation*}
In addition, by the definition of $\mu_{n}^\d$ it is easy to see that
\begin{eqnarray*}
&B_0^{p^*}\left(\mu_{n}^\d\right)^{p^*-1} \|F(x_{n}^\d)-y^\d\|^{p^*(r-1)} \le \mu_0^{p^*-1} \|F(x_{n}^\d)-y^\d\|^r.
\end{eqnarray*}
Then combining with these two inequalities with (\ref{1111}), we thus obtain
\begin{equation}
\begin{aligned}
&D_{\xi_{n+1}^\d} \Theta(\hat x, x_{n+1}^\d)-D_{\xi_{n}^\d}\Theta(\hat x, x_{n}^\d)\nonumber\\
&\leq (2 c_0)^{1-p^*}(3B_0)^{p^*}(\mu_{n}^\d )^{p^*}\| F(x_{n}^\d)-y^\d\|^{p^*(r-1)}
-\mu_n^\d \left(1-3\eta-\frac{3(1+\eta)}{\tau}\right) \|F(x_n^\d)-y^\d\|^{r}\\
&\leq 3^{p^*}(2 c_0)^{1-p^*}\mu_{n}^\d(\mu_0 )^{p^*-1}\| F(x_{n}^\d)-y^\d\|^{r}
-\mu_n^\d \left(1-3\eta-\frac{3(1+\eta)}{\tau}\right) \|F(x_n^\d)-y^\d\|^{r}\\
&\le - c_1\mu_{n}^\d \|F(x_{n}^\d)-y^\d\|^r,
\end{aligned}
\end{equation}
i.e., the error is decreasing.
To show $x_{n+1}^\d \in B_{2\rho}(x_0)$, we first use the above inequality with $\hat x =x^\dag$ and (\ref{5.11.1}) to obtain
\begin{equation*}
D_{\xi_{n+1}^\d}\Theta(x^\dag, x_{n+1}^\d) \le D_{\xi_0}\Theta(x^\dag, x_0) \le c_0 \rho^p.
\end{equation*}
In view of (\ref{pconv}), we then have  $\|x_{n+1}^\d-x^\dag\|\le \rho$ and $\|x^\dag-x_0\|\le \rho$.
Consequently $x_{n+1}^\d\in B_{2\rho}(x_0)$.

We next show $n_\d<\infty$.
According to the definition of $n_\d$, for any $n<n_\d$ such that $\|F(x_{n}^\d)-y^\d\|>\tau \d$.
Then there holds
$$\mu_{n}^\d =\mu_0 B_0^{-p} \|F(x_{n}^\d)-y^\d\|^{p-r},$$
and
\begin{equation*}
\sum_{n=0}^{n_\d} \mu_{n}^\d \|F(x_{n}^\d)-y^\d\|^r \ge \mu_0 B_0^{-p} \|F(x_{n}^\d)-y^\d\|^p
>\mu_0 B_0^{-p}(\tau\d)^p.
\end{equation*}
By summing (\ref{eq5.11}) over $n$ from $n=0$ to $n=m$ for any $m<n_\d$ and using the above inequality we obtain
$$c_1 \mu_0 B_0^{-p}(\tau\d)^p (m+1) \le D_{\xi_0}\Theta(\hat x, x_0).$$
Since this is true for any $m<n_\d$, it follows that $n_\d<\infty$.
\end{proof}
When the iteration (\ref{HPICP}) is applied to the exact data, i.e., using $y$ instead of $y^\d$ in (\ref{HPICP}), we will drop the superscript $\d$ in all the quantities involved, for instance, we will write $\xi_n^\d$ as $\xi_n$, $x_n^\d$ as
$x_n$, and so on.
Observing that
\begin{equation*}
\mu_{n} \|F(x_{n})-y\|^r =\mu_0 B_0^{-p}\|F(x_{n})-y\|^p.
\end{equation*}
The proof of Proposition \ref{L5.1} in fact shows that, under Assumption \ref{A1}, if
$$c_2:=1-3\eta-3\left(\frac{3\mu_0}{2c_0}\right)^{\frac{1}{p-1}}>0,$$
then
$$x_{n}\in B_{2\rho}(x_0) \qquad \forall n\ge 0,$$
and for any solution $\hat x$ of (\ref{Nonlinear}) in $B_{2\rho}(x_0)\cap D(\Theta)$ and all n there hold
\begin{eqnarray}
&&D_{\xi_{n+1}}\Theta(\hat x, x_{n+1})  \le D_{\xi_n}\Theta(\hat x, x_n),\label{eq5.10.1}\\
&&c_2\mu_0 B_0^{-p} \|F(x_{n})-y\|^p  \le D_{\xi_n}\Theta(\hat x, x_n)
-D_{\xi_{n+1}}\Theta(\hat x, x_{n+1}). \label{eq5.11.1}
\end{eqnarray}
These two inequalities imply immediately that
\begin{equation}\label{eq5.12.1}
\lim_{n\rightarrow \infty}\|F(x_{n})-y\|^p =0.
\end{equation}

As the first step toward the proof of convergence on $x_{n_\d}^\d$, we need to derive some convergence results on the sequences $\{x_n\}$ and $\{\xi_n\}$. This will be achieved by the following proposition which gives a general convergence criterion on any sequences $\{x_n\}\subset\X$ and $\{\xi_n\}\subset\X^*$ satisfying certain conditions.

\begin{lemma}\label{L5.21}
Let all the conditions in Proposition \ref{L5.1} hold.
For the sequences $\{\xi_n\}$ and $\{x_n\}$ defined by iteration (\ref{HPICP}) with exact data, there exists a solution $x_*\in B_{2\rho}(x_0)\cap D(\Theta)$ of (\ref{Nonlinear}) such that
$$\lim_{n\rightarrow \infty} \|x_n-x_*\|=0 ~~~ \text{and }~~ \lim_{n\rightarrow \infty} D_{\xi_n}\Theta(x_*, x_n)=0.$$
If in addition $\N(F'(x^\dag))\subset \N(F'(x))$ for all $x\in B_{2\rho}(x_0)$, then $x_*=x^\dag$.
\end{lemma}
\begin{proof}
We first show that there is a strictly increasing subsequence $\{n_k\}$ of integers such
that $\{x_{n_k}\}$ is convergent. To this end, let $$R_n:= \|y-F(x_{n})\|^p.$$
%We will use Proposition \ref{general} to complete the proof.
%By the definition of  $\{\xi_n\}$ and $\{x_n\}$ we have $\xi_n\in \p \Theta(x_n)$.
%The sequence $\{D_{\xi_n}\Theta(\hat x, x_n)\}$ is monotonically decreasing, and $$\lim_{n\rightarrow \infty} \|F(x_n)-y\|=0.$$
%Therefore, in order to derive the convergence result, it suffices to show that there exists a strictly increasing subsequence $\{n_k\}$ such that for any solution $\hat x$ of (\ref{Nonlinear}) and any $l<k$ there holds
% \begin{equation}\label{eq5.4}
%\left|\l \xi_{n_k}-\xi_{n_l}, x_{n_k}-\hat x\r\right| \le C\left(D_{\xi_{n_l}}\Theta(\hat x, x_{n_l}) -D_{\xi_{n_k}}\Theta (\hat x, x_{n_k})\right)
%\end{equation}p
It follows from (\ref{eq5.12.1}) that
\begin{equation}\label{eq5.1}
\lim_{n\rightarrow \infty} R_n =0.
\end{equation}
Moreover, if $R_n=0$ for some $n$, then $y=F(x_{n})$.
Consequently it follows from the definition of the method that $x_{m}=x_n$  for all $m\geq n$.
Therefore
\begin{equation}\label{eq5.2}
R_n=0  \Longrightarrow R_m=0~\text{for~all}~m\ge n.
\end{equation}
In view of (\ref{eq5.1}) and (\ref{eq5.2}), we can introduce a subsequence $\{n_k\}$ by setting $n_0=0$ and letting $n_k$, for each $k\ge 1$, be the first integer satisfying
$$n_k\ge n_{k-1}+1 ~~\text{and}~~ R_{n_k} \le R_{n_{k-1}}.$$
For such chosen strictly increasing sequence $\{n_k\}$ we have
\begin{equation}\label{eq5.3}
R_{n_k} \le R_n, \qquad 0\le n<n_k.
\end{equation}
Now for any $l < k$, we consider $D_{\xi_{n_l}}\Theta(x_{n_k}, x_{n_l})$ as
\begin{equation}\label{2.28}
D_{\xi_{n_l}}\Theta(x_{n_k}, x_{n_l}) = D_{\xi_{n_l}}\Theta(\hat{x}, x_{n_l})-D_{\xi_{n_k}}\Theta(\hat{x}, x_{n_k})
+ \l \xi_{n_k}-\xi_{n_l}, x_{n_k}-\hat x\r.
\end{equation}
where
\begin{equation}\label{eq5.3}
  \begin{aligned}
|\l \xi_{n_k}-\xi_{n_l}, x_{n_k}-\hat x\r|
&= \left|\sum_{n=n_l}^{n_k-1}\l \xi_{n+1}-\xi_{n}, x_{n_k}-\hat x\r\right|\\
&= \left|\sum_{n=n_l}^{n_k-1}\mu_{n} \l  2 r_n-\nu_n J_{r}^{\Y}(T_nJ_{s^*}^{*}T_n^* r_n), T_n(x_{n_k}-\hat x)\r \right|
  \end{aligned}
\end{equation}
Using the nonlinear condition on $F$ it is easy to obtain
\begin{equation*}
\begin{aligned}
 \|T_n(x_{n_k}-\hat x)\| &\le \|T_n(x_{n}-\hat x)\| +\|T_n(x_{n_k}-x_{n})\|\\
& \le (1+\eta) \left(\|F(x_{n})-y\|+ \|F(x_{n_k})-F(x_{n})\|\right)\\
&\le (1+\eta) \left(2 \|F(x_{n})-y\|+\|F(x_{n_k})-y\|\right).
\end{aligned}
\end{equation*}
Therefore, by using the property of the duality mapping $J_r^{\Y}$, we have
\begin{equation*}\label{eq5.6}
\begin{aligned}
|\l \xi_{n+1}-\xi_n, x_{n_k}-\hat x\r|
&\leq\mu_{n} |\l2 r_n-\nu_n^\delta J_{r}^{\Y}(T_nJ_{s^*}^{*}T_n^* r_n), T_n(x_{n_k}-\hat x)\r|\\
&\le 2\mu_{n} \|r_n\|\|T_n(x_{n_k}-\hat x)\|+ \mu_n\nu_n\|J_r^{\Y}(T_nJ_{s^*}^{*}T_n^*r_n)\|\|T_n(x_{n_k}-\hat x)\|\\
&\le 3 \mu_{n} \|F(x_{n})-y\|^{r-1} \|T_n(x_{n_k}-\hat x)\| \\
&\le 9(1+\eta)\mu_0 B_0^{-p} \|F(x_{n})-y\|^{p}.\\
\end{aligned}
\end{equation*}
Then combining this with the inequality (\ref{eq5.11.1}), we can derive that
\begin{equation}\label{eq5.6}
\begin{aligned}
|\l \xi_{n_k}-\xi_{n_l},x_{n_k}-\hat x\r|
&\le\sum_{n=n_l}^{n_k-1} |\l \xi_{n+1}-\xi_n, x_{n_k}-\hat x\r|\\
&\le 9(1+\eta)\mu_{0} B_0^{-p} \sum_{n=n_l}^{n_k-1} \|F(x_{n})-y\|^p\\
&\le C\left(D_{\xi_{n_l}}\Theta(\hat x, x_{n_l}) -D_{\xi_{n_k}}\Theta (\hat x, x_{n_k})\right)
\end{aligned}
\end{equation}
with $C=9(1+\eta)B_0^{-p}/c_2>0$.
Thus we have from \eqref{2.28} and \eqref{eq5.6} that
\begin{equation}\label{eq3.16}
D_{\xi_{n_l}}\Theta(x_{n_k}, x_{n_l}) \le (1+C)\left(D_{\xi_{n_l}}\Theta(\hat x, x_{n_l}) -D_{\xi_{n_k}}\Theta (\hat x, x_{n_k})\right).
\end{equation}
By the monotonicity of $D_{\xi_{n}}\Theta(\hat{x}, x_{n})$, we obtain  that $D_{\xi_{n_l}}\Theta(x_{n_k}, x_{n_l}) \to 0$ as $k,l\to \infty$.
By the p-convexity of $\Theta$ we can conclude that $\{x_{n_k}\}$ is a Cauchy sequence in $\X$ and thus $x_{n_k}\to x_*$ as $k\to \infty$ for some $x_*\in B_{2\rho}(x_0)\subset \X$.

Next we show that $x_*\in D(\Theta)$.
We use $\xi_{n_l}\in \p \Theta(x_{n_l})$ to obtain
\begin{equation}\label{n}
\begin{aligned}
\Theta(x_{n_k})&= \Theta(x_{n_l}) +\l \xi_{n_l}, x_{n_k}-x_{n_l}\r + D_{\xi_{n_l}} \Theta(x_{n_k}, x_{n_l})\\
&\leq \Theta(x_{n_l}) +\l \xi_{n_l}, x_{n_k}-x_{n_l}\r +(1+C)D_{\xi_{n_l}}\Theta(\hat{x}, x_{n_l})
\end{aligned}
\end{equation}
Since $x_{n_k}\rightarrow x_*$ as $k\rightarrow \infty$, by using the lower semi-continuity of $\Theta$ we obtain
\begin{equation*}
  \begin{aligned}
\Theta(x_*)\le \liminf_{k\rightarrow\infty} \Theta(x_{n_k})
\leq
\Theta(x_{n_l}) +\l \xi_{n_l}, x_*-x_{n_l}\r + (1+C) D_{\xi_{n_l}} \Theta(\hat{x}, x_{n_l})
<\infty
\end{aligned}
\end{equation*}
This implies that $x_*\in D(\Theta)$.

Furthermore, in order to derive the convergence in Bregman distance, we take $k\rightarrow \infty$, and use (\ref{eq3.16}) and $x_{n_k}\rightarrow x_*$  to  derive for $l<k$ that
\begin{equation*}\label{eq3.18}
  \begin{aligned}
D_{\xi_{n_l}}\Theta(x_{*}, x_{n_l})
 \leq \liminf_{k\to \infty}D_{\xi_{n_l}}\Theta(x_{n_k}, x_{n_l})\leq \liminf_{k\to \infty}(1+C)(D_{\xi_{n_l}}\Theta(x_*, x_{n_l})-\epsilon_0 )\\
\end{aligned}
\end{equation*}
where $\epsilon_0 := \lim_{n\to \infty} D_{\xi_{n}}\Theta(x_{*}, x_{n}) $, whose existence is guaranteed by the monotonicity of
$D_{\xi_{n}}\Theta(x_{*}, x_{n}) $.
Since the above inequality holds for all $l$, by letting $l\to \infty$ we can obtain $\epsilon_0 \leq 0$.
Therefore, we derive $\lim_{n\to \infty} D_{\xi_{n}}\Theta(x_{*}, x_{n}) =\epsilon_0 =0$.

Finally we show that $x_*=x^\dag$. We  have
\begin{equation}\label{eq2.38}
  \begin{aligned}
D_{\xi_0}\Theta(x_{n_k}, x_0)&= D_{\xi_0}\Theta(x^\dag, x_0)-D_{\xi_{n_k}}\Theta(x^\dag, x_{n_k})+\l \xi_{n_k}-\xi_0, x_{n_k}-x^\dag\r\\
&\le D_{\xi_0}\Theta(x^\dag, x_0)+\l \xi_{n_k}-\xi_0, x_{n_k}-x^\dag\r.
\end{aligned}
\end{equation}
By using (\ref{eq5.6}), for any $\varepsilon>0$  we can find $k_0$ such that
\begin{equation*}
\left|\l \xi_{n_k}-\xi_{n_{k_0}}, x_{n_k}-x^\dag\r\right| <\frac{\varepsilon}{2}, \qquad k\ge k_0.
\end{equation*}
We next consider $\l \xi_{n_{k_0}}-\xi_0, x_{n_k}-x^\dag\r$.
By the definition of $\xi_n$, we have
\begin{equation*}\label{n}
  \begin{aligned}
\xi_{n+1}-\xi_n=
-\mu_n(2T_n^* r_n-\nu_n^\delta T_n^*J_{r}^{\Y}(T_nJ_{s^*}^{*}T_n^* r_n))%T_n^* J_r^{\Y}(F(x_{n})-y).
\end{aligned}
\end{equation*}
Since $\X$ is reflexive and %$\N(F'(x^\dag)) \subset \N(F'(x_{n+\frac{1}{2}}))$,
 $\N(F'(x^\dag))\subset \N(F'(x))$ for all
$x\in B_{2\rho}(x_0)$,
we have
$\overline{\R(F'(x)^*)}\subset \overline{\R(F'(x^\dag)^*)}$ and $\xi_{n+1}-\xi_n \in \overline{{\mathcal R}(F'(x^\dag)^*)}$.
Then we can find $v_{n}\in\Y^*$ and $\beta_{n}\in\X^{*}$ such that
\begin{equation*}
 \xi_{n+1}-\xi_n=F'(x^\dag)^* v_{n} +\beta_{n} \quad \mbox{and} \quad \|\beta_{n}\|
\le \frac{\varepsilon}{3 B_1 n_{k_0}}, \quad 0\le n<n_{k_0},
\end{equation*}
where $B_1>0$ is a constant such that $\|x_n-x^\dag\|\le B_1$ for all $n$. Consequently
\begin{equation*}\label{n}
  \begin{aligned}
 \left|\l \xi_{n_{k_0}}-\xi_0, x_{n_k}-x^\dag\r \right|
&=\left|\sum_{n=0}^{n_{k_0}-1} \l \xi_{n+1}-\xi_n, x_{n_k}-x^\dag\r\right|\\
& =\left|\sum_{n=0}^{n_{k_0}-1}  \left[\l v_{n}, F'(x^\dag) (x_{n_k}-x^\dag)\r
+\l \beta_{n}, x_{n_k}-x^\dag\r \right]\right|\\
&\le \sum_{n=0}^{n_{k_0}-1} \left(\|v_{n}\| \|F'(x^\dag) (x_{n_k}-x^\dag)\|
+\|\beta_{n}\| \|x_{n_k}-x^\dag\|\right)\\
&\le (1+\eta) \sum_{n=0}^{n_{k_0}-1}  \|v_{n}\| \|F(x_{n_k})-y\| +\frac{\varepsilon}{3}.
\end{aligned}
\end{equation*}
Since $\|F(x_{n_k})-y\|\rightarrow 0$ as $n\rightarrow \infty$, we can find $k_1\ge k_0$ such that
\begin{equation*}
|\l \xi_{n_{k_0}}-\xi_0, x_{n_k}-x^\dag\r |<\frac{\varepsilon}{2}, \qquad \forall k\ge k_1.
\end{equation*}
Therefore $|\l \xi_{n_k}-\xi_0, x_{n_k}-x^\dag\r|<\varepsilon$ for all $k\ge k_1$. Since $\varepsilon>0$ is arbitrary,
we obtain $\lim_{k\rightarrow \infty} \l \xi_{n_k}-\xi_0, x_{n_k}-x^\dag\r=0$.
By taking $k\rightarrow \infty$ in \eqref{eq2.38} we obtain
\begin{equation*}
D_{\xi_0}\Theta(x_*, x_0)\le D_{\xi_0}\Theta(x^\dag, x_0).
\end{equation*}
According to the definition of $x^\dag$ we must have $D_{\xi_0}\Theta(x_*, x_0)=D_{\xi_0}\Theta(x^\dag, x_0)$.
A direct application of Lemma \ref{lemxdag} gives $x_*=x^\dag$.
\end{proof}

In order to use the above result to establish the convergence of Algorithm 1, we also need the following stability result.
\begin{theorem}[Stability analysis]\label{stability}
Let $\X$ be reflexive and let $\Y$ be uniformly smooth.
Let all the conditions in Proposition \ref{L5.1} hold.  Then for all $n\ge 0$ there hold
$$\xi_{n}^\d \rightarrow \xi_{n} ~~\text{and}~~ x_{n}^\d\rightarrow x_{n}, ~ \text{as }~\d\rightarrow 0.$$
\end{theorem}
\begin{proof}
The result is trivial for $n=0$.
We next assume that the result is true for some $n\ge 0$ and show that
$\xi_{n+1}^\d \rightarrow \xi_{n+1}$ and $x_{n+1}^\d\rightarrow x_{n+1}$ as $\d\rightarrow 0$.
We consider two cases.

\texttt{Case 1:} $F(x_{n})=y$.
In this case we have $\mu_{n}=0$ and $\lim\limits_{\d\rightarrow 0}\|F(x_{n}^\d) -y^\d\|= 0$ by the continuity of $F$.
Thus
$$\xi_{n+1}^\d-\xi_{n+1} =\xi_{n}^\d-\xi_{n} -\mu_{n}^\d {T_n^\d}^*\left(2 r_n^\d-\nu_n^\delta J_{r}^{\Y}({T_n^\d}J_{s^*}^{*}{T_n^\d}^* r_n^\d)\right)$$
which implies that
\begin{equation*}
\begin{aligned}
\|\xi_{n+1}^\d-\xi_{n+1}\|& \le\|\xi_{n}^\d-\xi_{n}\| +3 B_0\mu_{n}^\d \|F(x_{n}^\d)-y^\d\|^{r-1}\\
& \le\|\xi_{n}^\d-\xi_{n}\| +3\mu_0 B_0^{1-p} \|F(x_{n}^\d)-y^\d\|^{p-1}.
\end{aligned}
\end{equation*}
By the induction hypotheses, we then have $\lim\limits_{\d\rightarrow 0}\xi_{n+1}^\d= \xi_{n+1}$.
Consequently, by using the continuity of $\nabla \Theta^*$, we have $x_{n+1}^\d =\nabla \Theta^*(\xi_{n+1}^\d)\rightarrow
\nabla \Theta^*(\xi_{n+1}) =x_{n+1}$ as $\d\rightarrow 0$.

\texttt{Case 2:} $F(x_{n})\ne y$.
In this case we have $\|F(x_{n}^\d)-y^\d\|>\tau \d$ for small $\d\rightarrow 0$.
Therefore
$$\mu_{n}^\d=\mu_0 B_0^{-p} \|F(x_{n}^\d)-y^\d\|^{p-r} \rightarrow \mu_{n}=\mu_0 B_0^{-p} \|F(x_{n})-y\|^{p-r}$$
as $\d\to 0$.
By Assumption \ref{A1}(b) and the uniform smoothness of $\Y$, we know that $F$, $F'$ and $J_r^{\Y}$ are continuous.
It then follows from the induction hypotheses that
$\xi_{n+1}^\d\rightarrow \xi_{n+1}$ and $x_{n+1}^\d\rightarrow x_{n+1}$ as $\d\to 0$ using again the continuity of $\nabla \Theta^*$.
\end{proof}

We now apply the above results for proving the following convergence result, which shows the iteration (\ref{HPICP}) in combination with the discrepancy principle is a regularization method.

\begin{theorem}[Convergence analysis]\label{convergence}
Let $\X$ be reflexive and let $\Y$ be uniformly smooth.
Let Assumption \ref{A1} hold.
Let $\Theta: \X\to (-\infty, \infty]$ be proper, lower semi-continuous, and p-convex function satisfies (\ref{pconv}).
Assume that initial value $x_0$ and $\xi_0$ satisfies
$$D_{\xi_0} \Theta (x^\dag, x_0) \le c_0 \rho^p.$$
Then for $\{\xi_n^\d\}$ and $\{x_n^\d\}$ defined by Algorithm 1 with $\tau>1$ and $\mu_0>0$ satisfying
$$c_1:=1-3\eta-\frac{3(\eta+1)}{\tau}-3\left(\frac{3\mu_0}{2c_0}\right)^{\frac{1}{p-1}}>0.$$
there is a solution $x_*\in B_{2\rho}(x_0)\cap D(\Theta)$ of (\ref{Nonlinear}) such that
$$\lim_{\d\rightarrow 0} \|x_{n_\d}^\d-x_*\|=0~~\text{and}~~
\lim_{\d\rightarrow 0} D_{\xi_{n_\d}^\d}\Theta(x_*, x_{n_\d}^\d) =0.$$
If in addition $\N(F'(x^\dag))\subset \N(F'(x))$ for all $x\in B_{2\rho}(x_0)\cap D(F)$, then $x_*=x^\dag$.
\end{theorem}
\begin{proof}
See Theorem 3.9 in \cite{JinWang}.
\end{proof}

\section{Numerical examples}
In this section we present some numerical simulations with one-dimentional and two-dimensional cases to test the good performance of the proposed HPICP method with various choices of the convex function $\Theta$, in comparison with the existing Landweber iteration method with convex penalty (LICP).
Our simulations were done by using MATLAB R2010a on a Lenovo laptop with Intel Core i5-4200U CPU 2.30 GHz and 4.00 GB memory.

A key ingredient for HPICP method is the resolution of the minimization problem
\begin{equation}\label{eq:4.1}
x= \arg \min_{z\in \X} \left\{\Theta(z) -\l \xi, z\r\right\},~~\forall \xi\in\X^*.
\end{equation}
Next we give some discussion on the resolution of \eqref{eq:4.1} for various choices of $\Theta$ with $p=2$ as follows:

\emph{Case I: } Let $\X=L^2(\Omega)$ and the sought solution is partly sparse, we may consider the 2-convex function
\begin{equation}\label{eq:L1}
\Theta(x):= \frac{1}{2\beta} \int_\Omega |x(\omega)|^2 d\omega +\int_\Omega |x(\omega| d\omega
\end{equation}
with $\beta>0$.
The minimization of \ref{eq:4.1} for this case can be given explicitly by the following soft thresholding:
\begin{equation*}
\begin{aligned}
x&=\arg \min_{x\in L^2} \left\{\|x\|_{L^1} + \frac{1}{2\beta} \|x\|^2_{L^2}-\l \xi, x\r\right\}\\
&=\left\{\begin{array}{lll}
\beta(\xi(\omega)-1), & \mbox{ if } \xi(\omega) >1,\\
0, & \mbox{ if } |\xi(\omega)|\le 1,\\
\beta(\xi(\omega) +1), & \mbox{ if } \xi(\omega) <-1.
\end{array}\right.
\end{aligned}
\end{equation*}

\emph{Case II: } Let $\X=L^2(\Omega)$ and the sought solution is piecewise constant, we may consider the total variation like function
\begin{equation}\label{eq:TV}
\Theta(x):= \frac{1}{2\beta} \int_\Omega |x(\omega)|^2 d\omega +TV(x),
\end{equation}
with $\beta>0$.
The minimization of (\ref{eq:4.1}) for this case can be given explicitly as following:
\begin{equation*}
\begin{aligned}
x=\arg \min_{x\in L^2} \left\{\beta TV(x) + \frac{1}{2} \|x-\beta\xi\|^2_{L^2(\Omega)}\right\},
\end{aligned}
\end{equation*}
which is the well-known ROF model (see \cite{ROF}) in image denoising.
There are many efficient numerical solvers developed in the literature \cite{BT2009n,BT2009,cp11,TV1,TV2,TV3}; we use the fast iterative shrinkage-thresholding algorithm (FISTA) introduced from \cite{BT2009n,BT2009} in our numerical simulations.

We here consider the nonlinear model problem which consists of recovering the potential term in an elliptic equation.
Let $\Omega\subset\mathbb{R}^d(d=1,2)$ be an open bounded domain with a Lipschitz boundary $\Gamma$ and $f\in L^2(\Omega)$.
We consider the identification of the parameter $c$ in the equation
\begin{eqnarray}\label{Potential}
\begin{aligned}
\left\{\begin{array}{ll}
&-\triangle u+cu=f,~\textup{in}~\Omega,\cr
&\frac{\p u}{\p n}=0,~\textup{on}~\Gamma.
\end{array}\right.
\end{aligned}
\end{eqnarray}
We assume that the true potential $c^\dag$ is in $L^2(\Omega)$.
For each $c$ in the domain $D(F):=\left\{c\in L^{\infty}(\Omega):c\geq\bar{c}~\text{for some}~\bar{c}\geq0\right\}$, \eqref{Potential} has a unique solution $u=u(c)\in H^1(\Omega)$.
By the Sobolev embedding $H^1(\Omega)\hookrightarrow L^r(\Omega)$, we can define the nonlinear operator $F: \X=L^2(\Omega)\to \Y=L^r(\Omega)$ with $F(c)=u(c)$ for any $1<r<\infty$.
Hence we identify $c$ in the admissible set $D(F)$ from an $L^r$ measurement of $u$.
Recall that in the Banach space $L^r(\Omega)$ with $1<r<\infty$, the duality mapping $J_r: L^r(\Omega)\to L^{r^*}(\Omega)$ is given by
$$J_r(\upsilon):=|\upsilon|^{r-1}\text{sign}(\upsilon),~\upsilon\in L^r(\Omega).$$

We next will report numerical results to indicate the performance of HPICP method with various choices of the convex function $\Theta$ and the Banach spaces $\Y$.
The main computational cost stems from the numerical solutions of differential equations related to calculating the Fr$\acute{e}$chet  derivatives and their adjoint.
In order to carry out the computation, the forward operator was discretized using finite elements on a uniform grid (triangular, in the case of two dimentions), which is based on the shared Matlab code by Bangti Jin of \cite{Outliers}.
Given the true parameter $c^\dag(x)$, the simulated noise data $u^\d$ is generated by adding noise to the synthetic exact data $u^\dag=F(c^\dag)$ as follows
$$u^\d=u^\dag+\d \cdot n,$$
here $\d$ is the noise level and $n$ is the random variable obeying the standard normal distribution.
In addition, in order to measure the accuracy of solution more quantitatively, we employ the following relative error
$$\text{RE}=\frac{\|c-c^\dag\|}{\|c^\dag\|},$$
where $c$ represents the approximate solution.

In the following we implement the HPICP method using $c_0=\xi_0=0$ as the initial guess.
We take the step length $\nu_n = \|{T_n^\d}^* r_n^\d\|_{L^2}^2 / \|{T_n^\d}^*J_{r}^{\Y}(T_n^\d J_{s^*}^{*}{T_n^\d}^* r_n^\d)\|_{L^2}^2$ with
$$\mu_n^\d=\frac{\mu_0\|F(x_n^\d)-y^\d\|_{\Y}^{(p-1)r}}{\|{T_n^\d}^* r_n^\d\|_{L^2}^p},$$
here $\mu_0=(1-1/\tau)/\beta$.
To test the effects of $\beta$ for given convex penalty (\ref{eq:L1}) and (\ref{eq:TV}), we apply different choice for $\beta$ to perform the numerical computation.
We will later report the detailed numerical results recovered by our proposed method (HPICP) and the current existing method (LICP), respectively, including the required iteration number (i.e., $n_\d$), the computational time (i.e., time(s)) as well as the relative error (i.e., RE) between the true solutions and the regularized solutions.
\begin{figure}[htbp]
\centering
\begin{minipage}[b]{0.32\textwidth}
\centering
\includegraphics[width=1.8in]{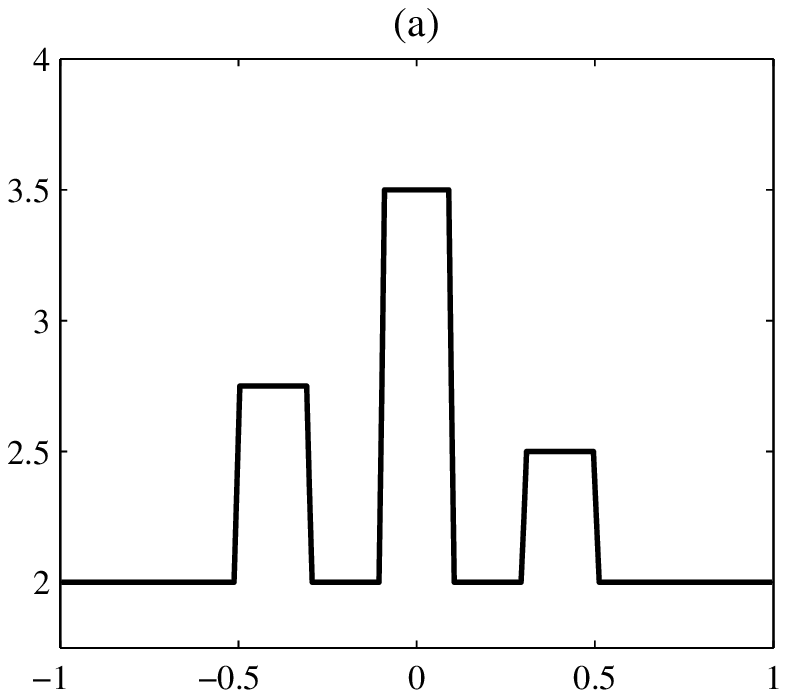}
\end{minipage}
\begin{minipage}[b]{0.32\textwidth}
\centering
\includegraphics[width=1.8in]{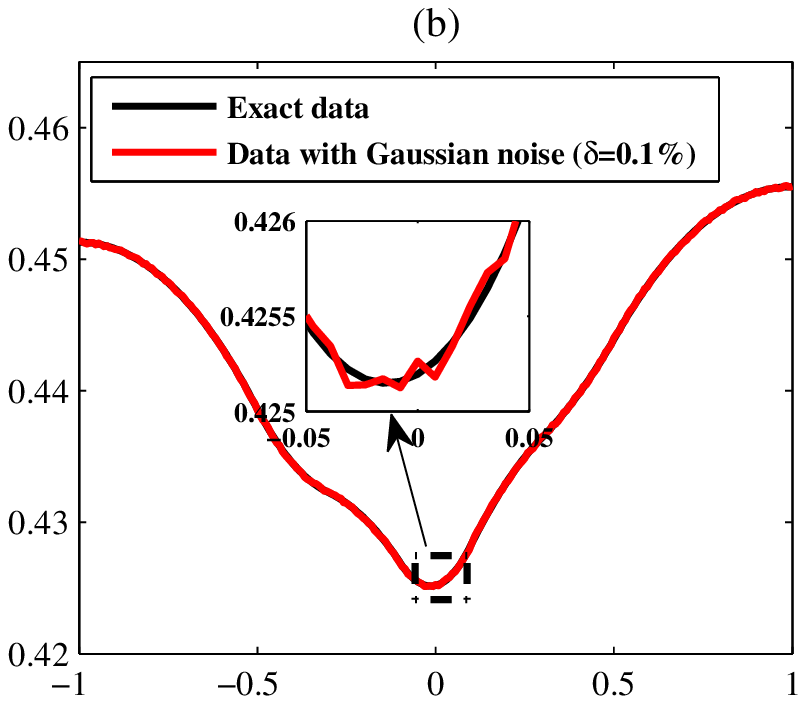}
\end{minipage}
\begin{minipage}[b]{0.32\textwidth}
\centering
\includegraphics[width=1.8in]{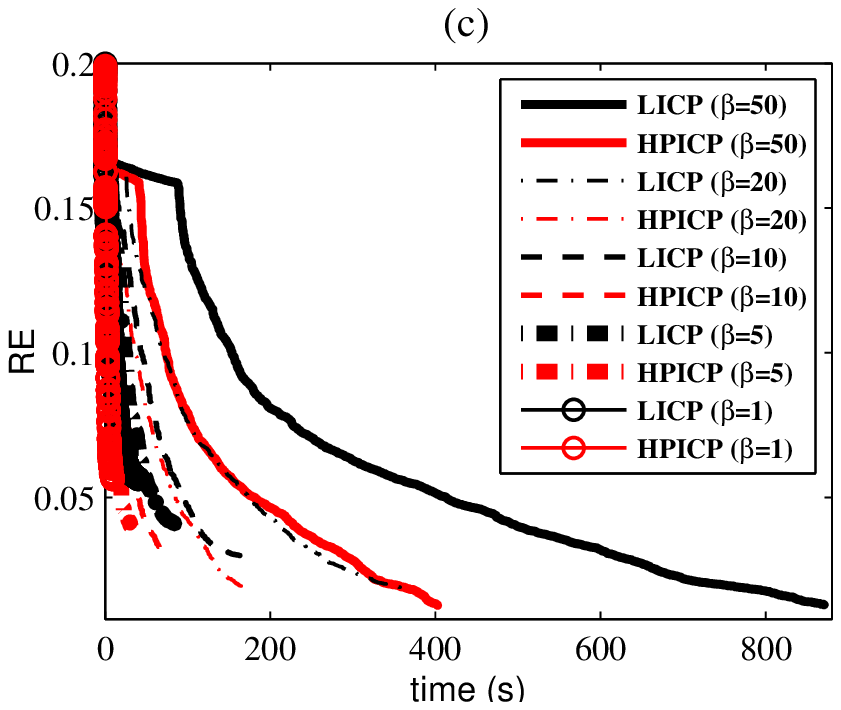}
\end{minipage}
\begin{minipage}[b]{0.32\textwidth}
\centering
\includegraphics[width=1.8in]{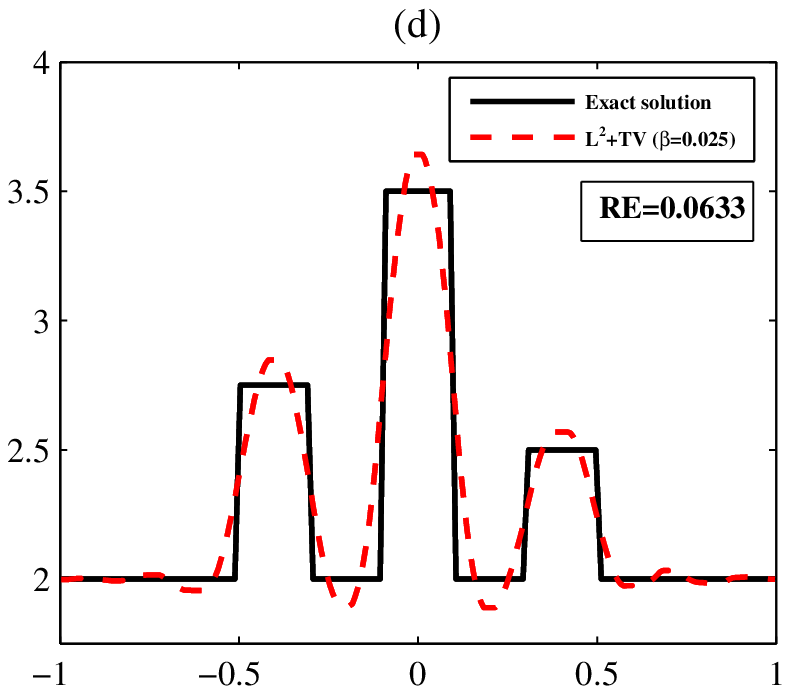}
\end{minipage}
\begin{minipage}[b]{0.32\textwidth}
\centering
\includegraphics[width=1.8in]{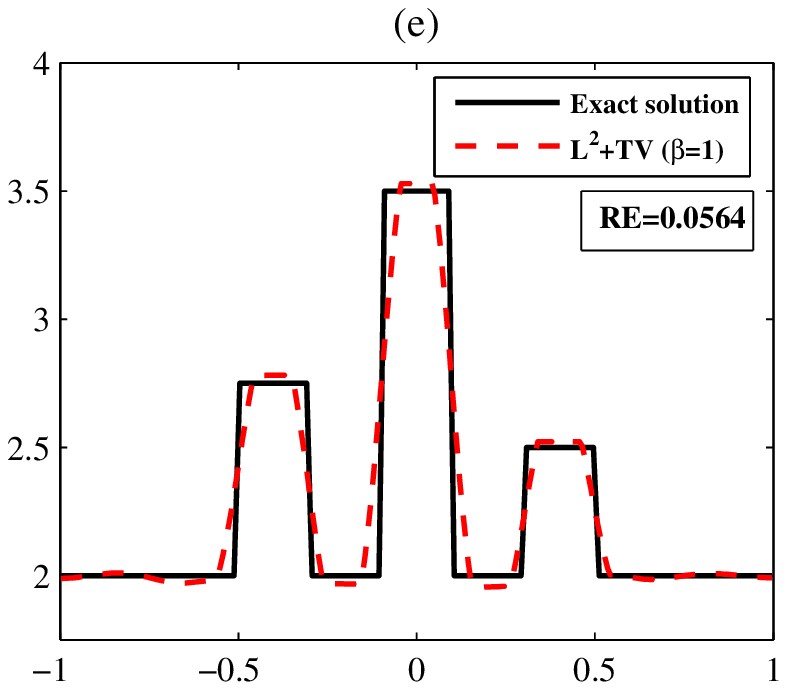}
\end{minipage}
\begin{minipage}[b]{0.32\textwidth}
\centering
\includegraphics[width=1.8in]{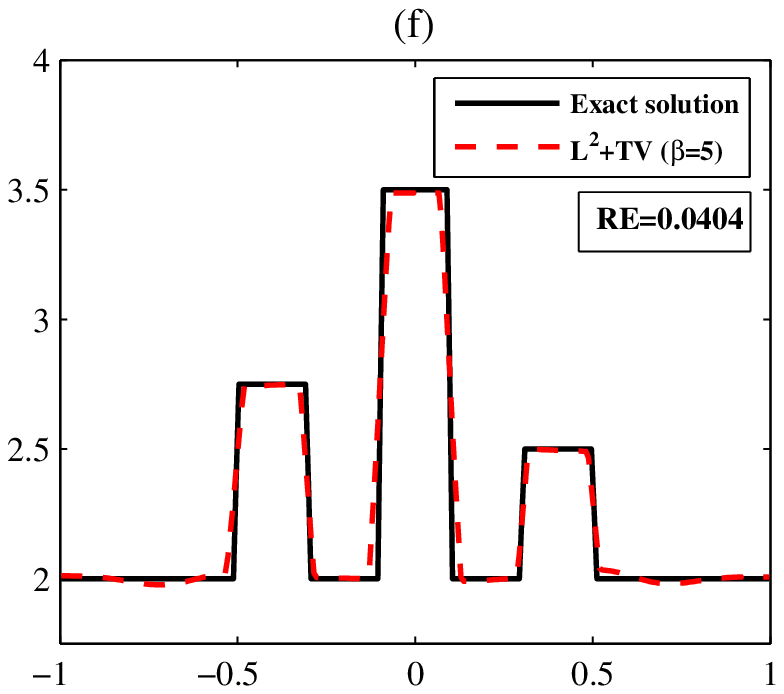}
\end{minipage}
\begin{minipage}[b]{0.32\textwidth}
\centering
\includegraphics[width=1.8in]{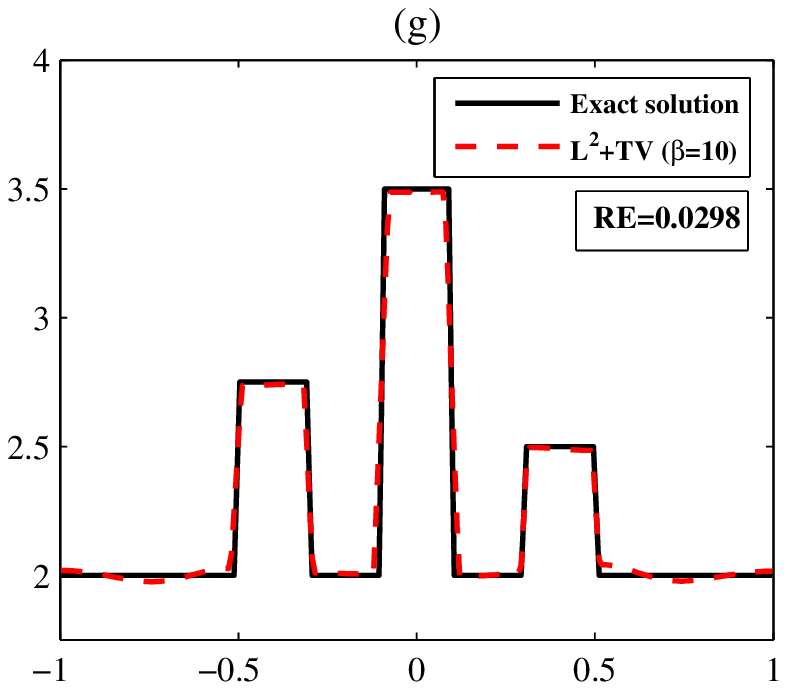}
\end{minipage}
\begin{minipage}[b]{0.32\textwidth}
\centering
\includegraphics[width=1.8in]{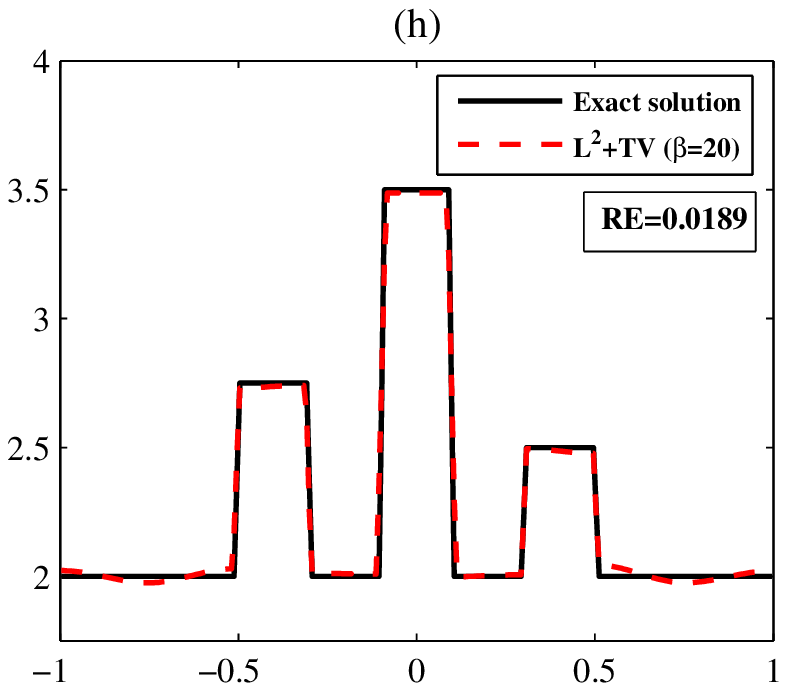}
\end{minipage}
\begin{minipage}[b]{0.32\textwidth}
\centering
\includegraphics[width=1.8in]{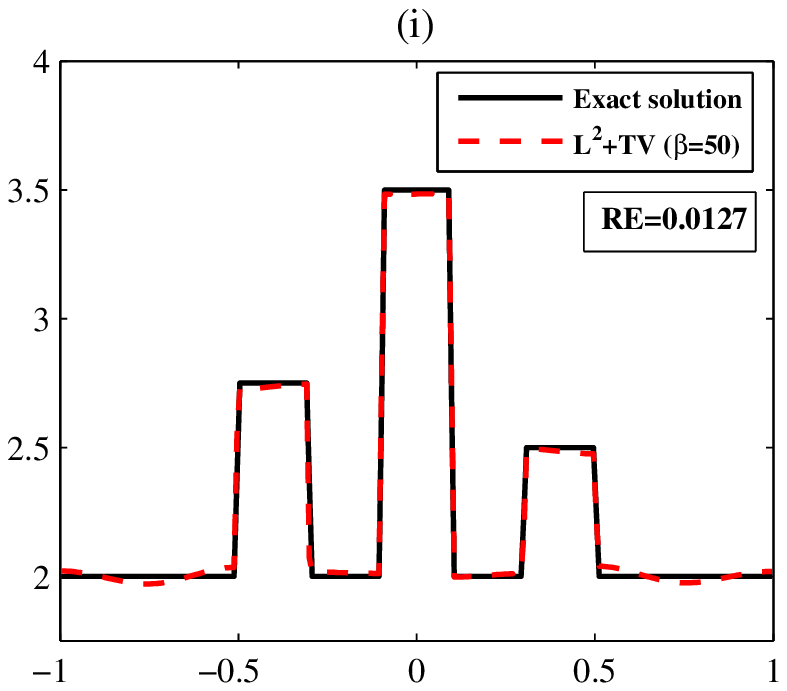}
\end{minipage}
\caption{(a) exact solution; (b) exact data and noisy data with Gaussian noise; (c) convergence behavior by HPICP and LICP, respectively; (d)-(i): reconstruction results by HPICP with different choice of $\beta$.}
\label{Figure 1}
\end{figure}

\emph{\textbf{Example 4.1} One-dimensional example.} We consider the one-dimensional problem (\ref{Potential}) on the interval $\Omega=[-1,1]$ with the source term $f(x)=1$.
The mesh size is $h=1/N$ with the grid points number $N=256$.
Figure 1(a) shows the situation that the sought solution is piecewise constant, which is given by $$c^\dag=2+\frac{3}{4}\chi_{[-0.5,-0.3]}+\frac{3}{2}\chi_{[-0.1,0.1]}+\frac{1}{2}\chi_{[0.3,0.5]}.$$
For this case, we take $\Theta$ to be $L^2+TV$ regularization functional defined in \eqref{eq:TV}.

We identify the true parameter $c^\dag(x)$ given in Figure 1(a) using Gaussian noisy data with $\d=0.1\%$ noise level shown in Figure 1(b).
We choose $\tau=1.1$ in the discrepancy principle.
The comparison of reconstructed results by HPICP and LICP with $r=2$ (i.e., $\Y=L^2[-1,1]$) are summarized in Table 1.
It can be seen from Table 1 that the regularized solutions by HPICP have the similar relative error qualities to those by LICP, but in less iteration number and computational time.
That is to say, our proposed HPICP method leads to a strongly decrease of the iteration numbers and the overall computational time can be significantly reduced.
In particular, for smaller $\beta$ the calculation times shows better performance of HPICP method under consideration.
On the other hand, it is clear that when relatively larger $\beta$ is used, more accurate reconstructed results can be obtained; however the computation could take longer time because the convexity of the minimization problem involved becomes weaker and hence more iteration steps are required to obtain an approximate minimizer within a certain accuracy.

\begin{table}[htbp]\sf \footnotesize% \scriptsize
\centering
\caption{Comparison of numerical results for Example 4.1 by HPICP and LICP with four different values of $\beta$ under the noise level $\delta=0.1\%$.}
\centering
    \begin{tabular}{|c|c|c|c|c|c|c|c|}
    \hline
    \diagbox{LICP}{\textbf{HPICP}}& $n_\d$&RE&time(s)\cr\hline
    $\beta=0.025$&\diagbox{24903}{ \textbf{1676}}&\diagbox{0.0638}{ \textbf{0.0633}}&\diagbox{35.7858}{\textbf{2.4785}}\cr\hline
    $\beta=~1$&\diagbox{9931}{ \textbf{2150}}&\diagbox{0.0564}{ \textbf{0.0564}}&\diagbox{40.5642}{\textbf{9.3851}}\cr\hline
    $\beta=~5$&\diagbox{10072}{\textbf{3460}}&\diagbox{0.0405}{ \textbf{0.0404}}&\diagbox{88.0081}{\textbf{30.8745}}\cr\hline
    $\beta=10$&\diagbox{14573}{\textbf{6991}}&\diagbox{0.00299}{ \textbf{0.0298}}&\diagbox{163.7898}{\textbf{78.9212}}\cr\hline
    $\beta=20$&\diagbox{25322}{ \textbf{12137}}&\diagbox{0.0190}{ \textbf{0.0189}}&\diagbox{357.6894}{\textbf{174.4382}}\cr\hline
    $\beta=50$&\diagbox{55143}{ \textbf{26733}}&\diagbox{0.0129}{ \textbf{0.0127}}&\diagbox{853.1009}{\textbf{403.0396}}\cr\hline
    \end{tabular}
\end{table}

In order to visibly illustrate the convergence behavior of both methods, we draw the curves from RE vs. time(s) with various parameter $\beta$ in Figure 1(c).
It is clear that the relative errors of HPICP are consistently lower than those of LICP.
As can be expected, the convergence rate of the proposed HPICP is significantly accelerated compared to that of LICP.
We then in Figure 1(d)-(i) plot the corresponding regularized solutions by HPICP for some selected values of $\beta$ to further visualize the performance.
Since the results by LICP have the similar qualities to those by HPICP, we here do not list them.
We observe that all the locations of the bumps are correctly identified with appropriate value of $\beta$, and their magnitudes are also reasonable.

In addition, in order to test the robustness of the proposed method (HPICP) to noise, four various noise level are added to the generated exact data, respectively.
For each noise level, we summarized detailed computational results in Table 2.
As can be expected, HPICP shows the favorable robustness.
We observe from Table 2 that with the increase of the noise levels the relative errors increase, which indicates that the accuracy of the measurement data has an effect on the reconstructed solution quality.

\begin{table}[htbp] \sf \footnotesize
\centering
\caption{Comparison of numerical results for Example 4.1 by HPICP with $\beta=20$ at four different noise levels.}
\centering
    \begin{tabular}{|c|c|c|c|}
    \hline
    Noise level&$n_\d$ & RE&Time(s) \cr\hline
    $\delta=~~~1\%$&1744&0.0999&25.3991\cr\hline
    $\delta=~0.5\%$&3081&0.0760&46.5127\cr\hline
    $\delta=~0.1\%$&12137&0.0189&174.4382\cr\hline
    $\delta=0.05\%$&19550&0.0178&267.8776\cr\hline
    \end{tabular}
\end{table}

Finally, we present a graphical demonstration of the effect of using Banach spaces setting in our considerations.
Therefore we choose the Banach space $\Y=L^r[-1,1]$.
Figure 2(a) shows the plot of the noisy data that contains a few data points, called outliers, which are highly inconsistent with other data points.
$L^r$ misfit data terms with $r>1$ close to 1 are especially suitable for the outliers noise, see \cite{Outliers,Outliers1}.
For fixed $\beta=20$ and under the above outliers data, Figure 2(b)-(d) present the reconstruction results with $r=1.05$, $r=1.5$ and $r=2$, respectively.
It can be seen that the method with small $r$ is robust enough to prevent being affected by outliers.
%Consequently, the numerical effort as well as the obtained regularized solutions of this numerical example show that our proposed HPICP method is an interesting alternative to Tikhonov-type regularization based on Banach space norms data fitting.

\begin{figure}[htbp]
\centering
\begin{minipage}[b]{0.425\textwidth}
\centering
\includegraphics[width=2.35in]{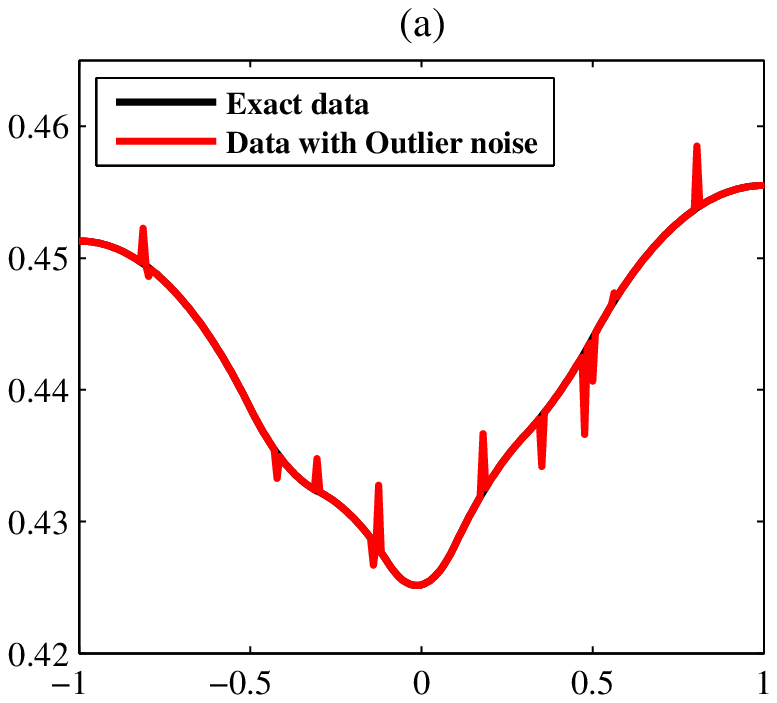}
\end{minipage}
\begin{minipage}[b]{0.425\textwidth}
\centering
\includegraphics[width=2.35in]{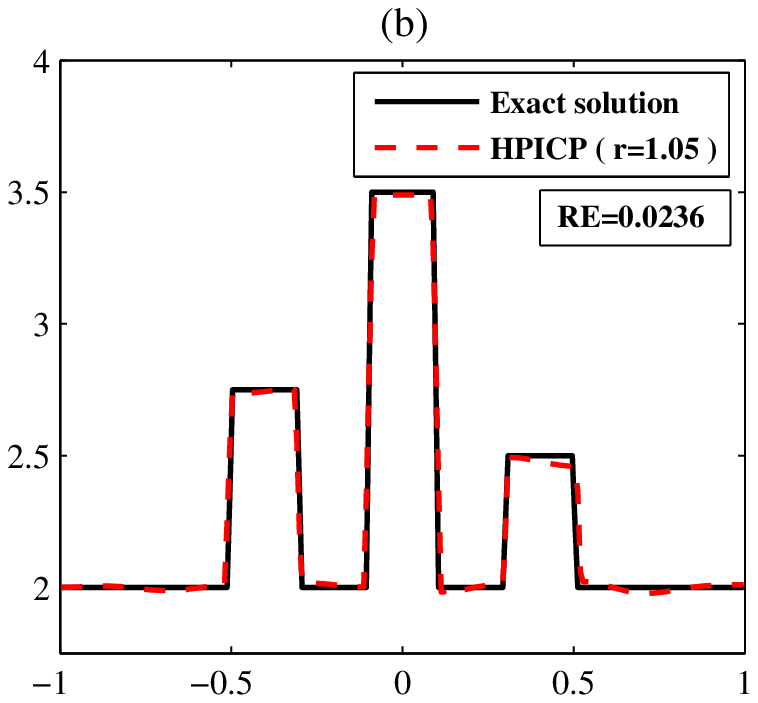}
\end{minipage}
\begin{minipage}[b]{0.425\textwidth}
\centering
\includegraphics[width=2.35in]{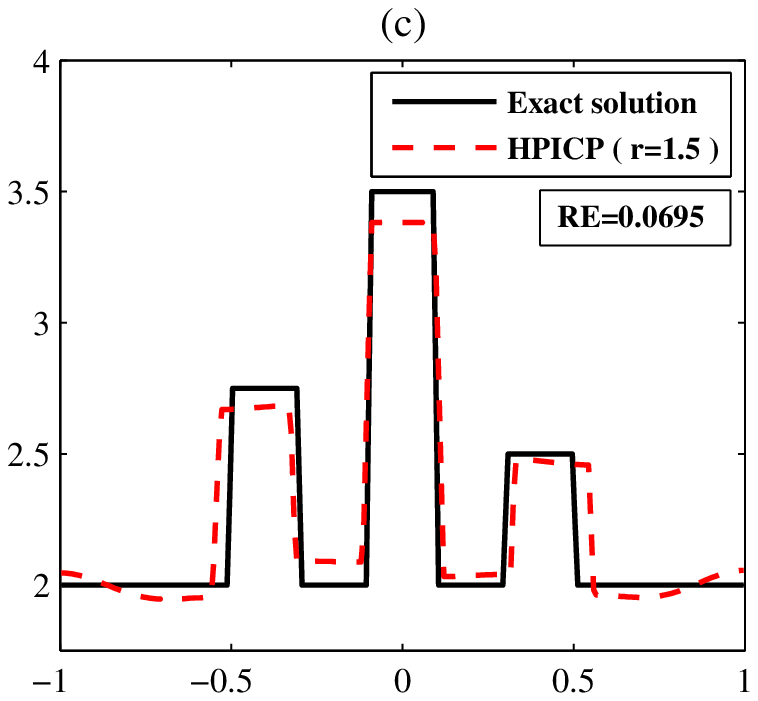}
\end{minipage}
\begin{minipage}[b]{0.425\textwidth}
\centering
\includegraphics[width=2.35in]{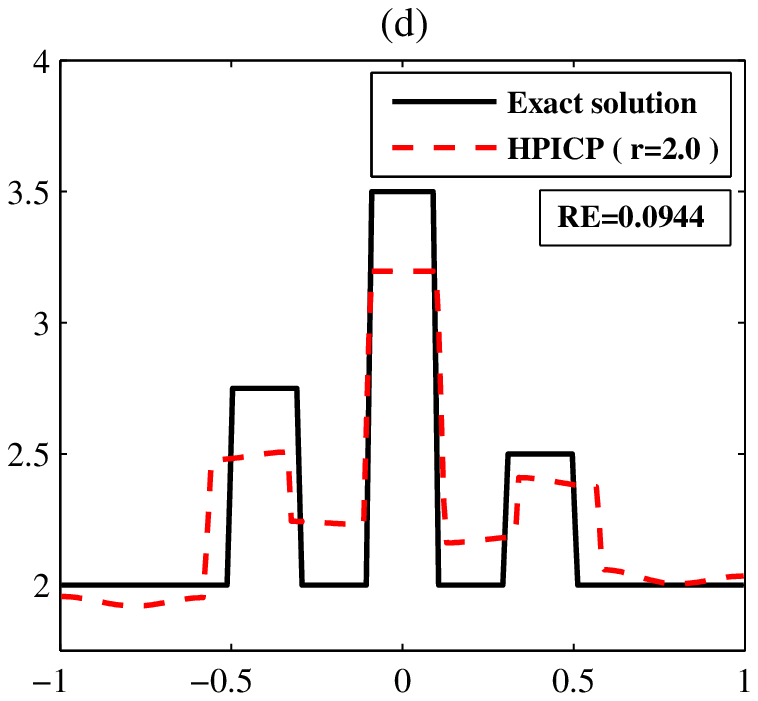}
\end{minipage}
\caption{Reconstruction obtained by HPICP method with different $r$ for fixed $\beta=20$.}
\label{Figure 2}
\end{figure}

\emph{\textbf{Example 4.2} Two-dimensional example.} Here, we consider the two-dimensional problem on the unit square, i.e., $\Omega=[-1,1]\times[-1,1]$, $f(x,y)=1$, and
$$c^\dag(x,y)=1+\cos(\pi x)\cos(\pi y)\chi_{|(x,y)|_\infty<1/2},$$
see Figure 3(a).
For this situation, we take $\Theta$ to be $L^2+L^1$ regularization functional defined in (\ref{eq:L1}).
To obtain exact data $u$ and noisy data $u^\d$, the mesh size for the forward solution is 7938 triangulation elements.
We choose $\tau=2.1$ in the discrepancy principle.
The more detailed comparison of the solution by HPICP and LICP with $\beta=1$ under five different noise levels are summarized in Table 3.
It can be seen that both methods provide regularized solutions of similar quality, but the iteration number and the computational time can be significantly reduced when the homotopy perturbation technique are used.
The results also indicate that both methods show the favorable robustness.
Furthermore, we in Figure 3(b)-(d) depict the numerical results of HPICP with $\beta=1$ under three different noise levels, respectively.
The numerical results of LICP have the similar qualities, and thus not shown here.
It is clear that the solution accurately captures the shape as well as the magnitude of the potential $c^\dag$, and thus represents a good approximation.

\begin{table}[htbp]\sf\footnotesize%\scriptsize
\centering
\caption{Comparison of numerical results for Example 4.2 by HPICP and LICP with $\beta=1$ under five different noise levels.}
\centering
    \begin{tabular}{|c|c|c|c|c|c|c|c|}
    \hline
    \diagbox{LICP}{\textbf{HPICP}}& $n_\d$ & RE &
    Time(s)\cr\hline
    $\delta=~~~1\%$&\diagbox{356}{\textbf{243}}&\diagbox{0.0195}{ \textbf{0.0195}}&\diagbox{28.8562}{\textbf{18.5048}}\cr\hline
    $\delta=~0.5\%$&\diagbox{538}{\textbf{297}}&\diagbox{0.0149}{ \textbf{0.0146}}&\diagbox{43.3457}{\textbf{24.9653}}\cr\hline
    $\delta=~0.1\%$&\diagbox{1398}{\textbf{960}}&\diagbox{0.0053}{ \textbf{0.0052}}&\diagbox{113.7625}{\textbf{79.5354}}\cr\hline
    $\delta=0.05\%$&\diagbox{2439}{\textbf{1313}}&\diagbox{0.0040}{ \textbf{0.0040}}&\diagbox{198.4180}{\textbf{111.0972}}\cr\hline
    $\delta=0.01\%$&\diagbox{10857}{\textbf{6457}}&\diagbox{0.0020}{ \textbf{0.0020}}&\diagbox{868.6441}{\textbf{546.8676}}\cr\hline
    \end{tabular}
\end{table}

\begin{figure}[htbp]
\centering
\begin{minipage}[b]{0.425\textwidth}
\centering
\includegraphics[width=2.35in]{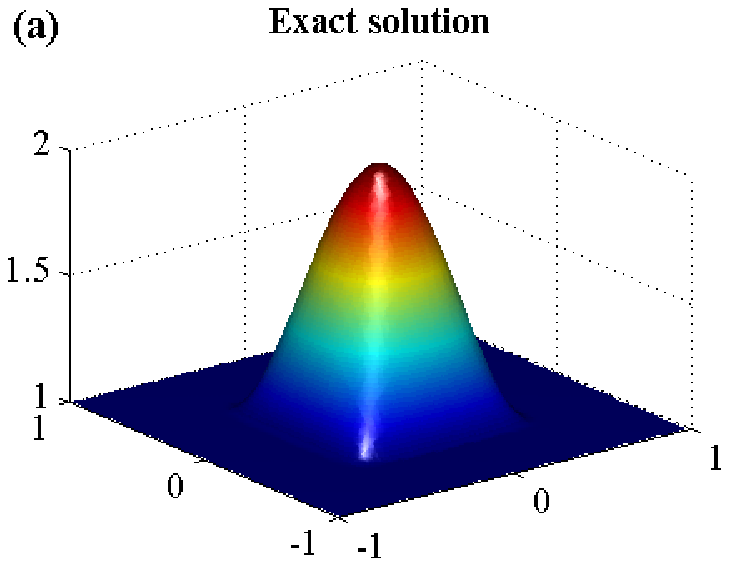}
\end{minipage}
\begin{minipage}[b]{0.425\textwidth}
\centering
\includegraphics[width=2.35in]{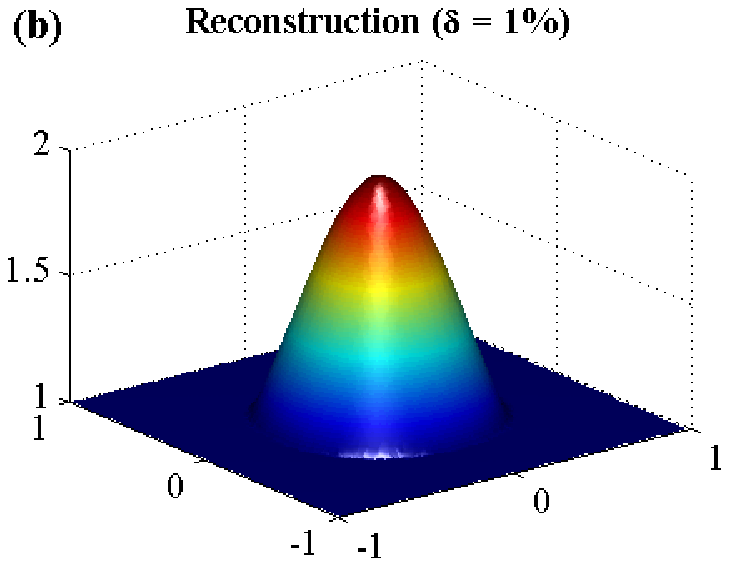}
\end{minipage}
\begin{minipage}[b]{0.425\textwidth}
\centering
\includegraphics[width=2.35in]{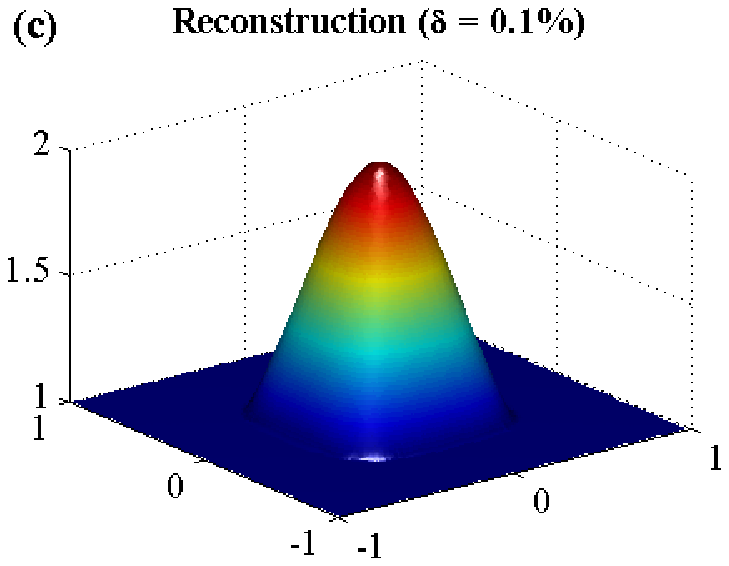}
\end{minipage}
\begin{minipage}[b]{0.425\textwidth}
\centering
\includegraphics[width=2.35in]{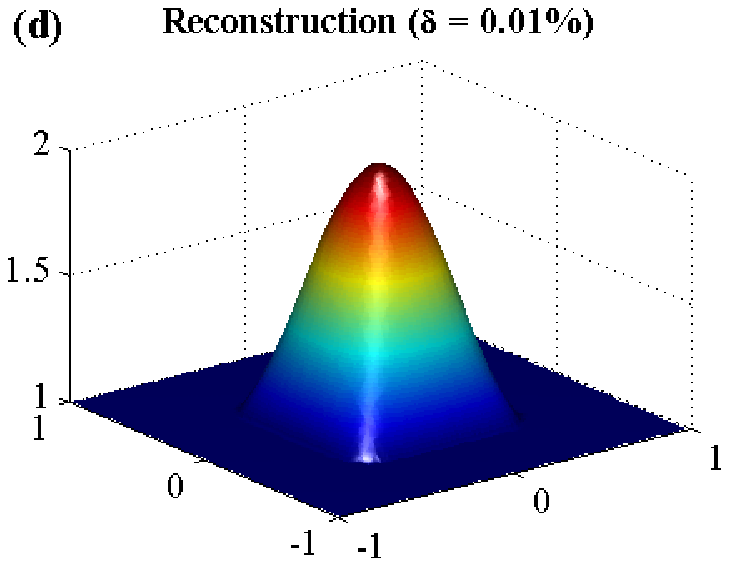}
\end{minipage}
\caption{Results for the 2d inverse potential problem by HPICP with $\beta=1$ at three different noise levels.}
\label{Figure 3}
\end{figure}

\section{Conclusion}
Motivated by chances of reducing numerical costs, this paper presented a novel iterative regularization approach with general uniformly convex penalty based on the homotopy perturbation technique for nonlinear ill-posed inverse problems in Banach spaces.
Convergence and regularization properties were shown, as well as some numerical examples were performed to illustrate the feasibility and effectiveness.
Compared with the existing Landweber iteration with general uniformly convex penalty, our approach reduces the overall computational time and improves the convergence rate.
How to extend the approach for a novel class of reconstruction schemes is the next work.

\section*{Acknowledgement}
%The authors would like to thank the editors and the anonymous referees for their constructive comments and suggestions, which have helped us to improve this manuscript.
%Moreover, the authors are also grateful to Dr. Qinian Jin (Australian National University, Australia) for some useful comments.
The authors are grateful to Dr. Qinian Jin (Australian National University, Australia) for some useful comments.
The work of Jing Wang is supported by the National Natural Science Foundation of China (NSFC) [grant number 11626092], Wei Wang by NSFC [grant number 11401257], Bo Han by NSFC [grant number 41474102].

\end{document}